\documentclass[12pt]{amsart}
\usepackage{mathrsfs}
\usepackage{amssymb}
\usepackage{verbatim}
\usepackage{hyperref}
\usepackage{color}
\usepackage{amsfonts}
\usepackage{mathrsfs}
\usepackage{amsmath}
\usepackage{amssymb}

\begin{document}
\newtheorem{theorem}{Theorem}
\newtheorem*{thmA}{Theorem A}
\newtheorem*{thmB}{Theorem B}
\newtheorem*{thmC}{Theorem C}
\newtheorem*{thmD}{Theorem D}
\newtheorem{proposition}[theorem]{Proposition}
\newtheorem{conjecture}[theorem]{Conjecture}
\def\theconjecture{\unskip}
\newtheorem{corollary}[theorem]{Corollary}
\newtheorem{lemma}[theorem]{Lemma}
\newtheorem{sublemma}[theorem]{Sublemma}
\newtheorem{observation}[theorem]{Observation}
\theoremstyle{definition}
\newtheorem{definition}{Definition}
\newtheorem{notation}[definition]{Notation}
\newtheorem{remark}[definition]{Remark}
\newtheorem{question}[definition]{Question}
\newtheorem{questions}[definition]{Questions}
\newtheorem{example}[definition]{Example}
\newtheorem{problem}[definition]{Problem}
\newtheorem{exercise}[definition]{Exercise}

\numberwithin{theorem}{section} \numberwithin{definition}{section}
\numberwithin{equation}{section}

\def\earrow{{\mathbf e}}
\def\rarrow{{\mathbf r}}
\def\uarrow{{\mathbf u}}
\def\varrow{{\mathbf V}}
\def\tpar{T_{\rm par}}
\def\apar{A_{\rm par}}

\def\reals{{\mathbb R}}
\def\torus{{\mathbb T}}
\def\heis{{\mathbb H}}
\def\integers{{\mathbb Z}}
\def\naturals{{\mathbb N}}
\def\complex{{\mathbb C}\/}
\def\distance{\operatorname{distance}\,}
\def\support{\operatorname{support}\,}
\def\dist{\operatorname{dist}\,}
\def\Span{\operatorname{span}\,}
\def\degree{\operatorname{degree}\,}
\def\kernel{\operatorname{kernel}\,}
\def\dim{\operatorname{dim}\,}
\def\codim{\operatorname{codim}}
\def\trace{\operatorname{trace\,}}
\def\Span{\operatorname{span}\,}
\def\dimension{\operatorname{dimension}\,}
\def\codimension{\operatorname{codimension}\,}
\def\nullspace{\scriptk}
\def\kernel{\operatorname{Ker}}
\def\ZZ{ {\mathbb Z} }
\def\p{\partial}
\def\rp{{ ^{-1} }}
\def\Re{\operatorname{Re\,} }
\def\Im{\operatorname{Im\,} }
\def\ov{\overline}
\def\eps{\varepsilon}
\def\lt{L^2}
\def\diver{\operatorname{div}}
\def\curl{\operatorname{curl}}
\def\etta{\eta}
\newcommand{\norm}[1]{ \|  #1 \|}
\def\expect{\mathbb E}
\def\bull{$\bullet$\ }

\def\xone{x_1}
\def\xtwo{x_2}
\def\xq{x_2+x_1^2}
\newcommand{\abr}[1]{ \langle  #1 \rangle}

\newcommand{\Norm}[1]{ \left\|  #1 \right\| }
\newcommand{\set}[1]{ \left\{ #1 \right\} }
\def\one{\mathbf 1}
\def\whole{\mathbf V}
\newcommand{\modulo}[2]{[#1]_{#2}}
\def \essinf{\mathop{\rm essinf}}
\def\scriptf{{\mathcal F}}
\def\scriptg{{\mathcal G}}
\def\scriptm{{\mathcal M}}
\def\scriptb{{\mathcal B}}
\def\scriptc{{\mathcal C}}
\def\scriptt{{\mathcal T}}
\def\scripti{{\mathcal I}}
\def\scripte{{\mathcal E}}
\def\scriptv{{\mathcal V}}
\def\scriptw{{\mathcal W}}
\def\scriptu{{\mathcal U}}
\def\scriptS{{\mathcal S}}
\def\scripta{{\mathcal A}}
\def\scriptr{{\mathcal R}}
\def\scripto{{\mathcal O}}
\def\scripth{{\mathcal H}}
\def\scriptd{{\mathcal D}}
\def\scriptl{{\mathcal L}}
\def\scriptn{{\mathcal N}}
\def\scriptp{{\mathcal P}}
\def\scriptk{{\mathcal K}}
\def\frakv{{\mathfrak V}}
\def\C{\mathbb{C}}
\def\R{\mathbb{R}}
\def\Rn{{\mathbb{R}^n}}
\def\Sn{{{S}^{n-1}}}
\def\M{\mathbb{M}}
\def\N{\mathbb{N}}
\def\Q{{\mathbb{Q}}}
\def\Z{\mathbb{Z}}
\def\F{\mathcal{F}}
\def\L{\mathcal{L}}
\def\S{\mathcal{S}}
\def\supp{\operatorname{supp}}
\def\dist{\operatorname{dist}}
\def\essi{\operatornamewithlimits{ess\,inf}}
\def\esss{\operatornamewithlimits{ess\,sup}}
\author{Zhengyang Li}
\address{Zhengyang Li \\
         School of Mathematical Sciences \\
         Beijing Normal University \\
         Laboratory of Mathematics and Complex Systems \\
         Ministry of Education \\
         Beijing 100875 \\
         People's Republic of China}
\email{zhengyli@mail.bnu.edu.cn}

\author{Qingying Xue}
\address{Qingying Xue\\
        School of Mathematical Sciences\\
        Beijing Normal University \\
        Laboratory of Mathematics and Complex Systems\\
        Ministry of Education\\
        Beijing 100875\\
        People's Republic of China}
\email{qyxue@bnu.edu.cn}

\thanks{The second author was supported partly by NSFC
(No. 11471041), the Fundamental Research Funds for the Central Universities (NO. 2014KJJCA10) and NCET-13-0065. \\ \indent Corresponding
author: Qingying Xue\indent Email: qyxue@bnu.edu.cn}


\date{May 2, 2016.}
\title[Bi-parameter Littlewood-Paley operators on product Hardy spaces]{\textbf{Boundedness of Bi-parameter} Littlewood-Paley operators on product Hardy space}
\maketitle

\begin{abstract}
Let $n_1,n_2\ge 1, \lambda_1>1$ and $\lambda_2>1$. For any $x=(x_1,x_2) \in \R^n\times\R^m$, let $g$ and $g_{\vec{\lambda}}^*$ be the bi-parameter Littlewood-Paley square functions defined by
\begin{align*}
g(f)(x)= \Big(\int_0^{\infty}\int_0^{\infty}|\theta_{t_1,t_2} f(x_1,x_2)|^2 \frac{dt_1}{t_1} \frac{dt_2}{t_2} \Big)^{1/2}, \hbox{\ \ and}
\end{align*}
$$
g_{\vec{\lambda}}^*(f)(x)
= \Big(\iint_{\R^{m+1}_{+}} \iint_{\R^{n+1}_{+}} \prod_{i=1}^2\Big(\frac{t_1}{t_i + |x_i - y_i|}\Big)^{n_i \lambda_i}
|\theta_{t_1,t_2} f(y_1,y_2)|^2 \frac{dy_1 dt_1}{t_1^{n+1}} \frac{dy_2 dt_2}{t_2^{m+1}} \Big)^{1/2},
$$
\noindent
where
$\theta_{t_1,t_2} f(x_1, x_2) = \iint_{\R^n\times\R^m} s_{t_1,t_2}(x_1,x_2,y_1,y_2)f(y_1,y_2) dy_1dy_2$. It is known that
the $L^2$ boundedness of bi-parameter $g$ and $g_{\vec{\lambda}}^*$ have been established recently by Martikainen, and Cao, Xue, respectively. In this paper, under certain structure conditions assumed on the kernel $s_{t_1,t_2},$ we show that both $g$ and $g_{\vec{\lambda}}^*$ are bounded from product Hardy space $H^1(\R^n\times\R^m)$ to $L^1(\R^n\times\R^m)$. As consequences, the $L^p$ boundedness of $g$ and $g_{\vec{\lambda}}^*$ will be obtained for $1<p<2$.\end{abstract}


\section{Introduction}

\subsection{Bi-parameter Littlewood-Paley operators}

The study of multi-parameter operators originated in the famous works of Fefferman and Stein \cite{F-S} on bi-parameter singular integral operators. Later on, Journ\'{e} \cite{Journe} peresented a multi-parameter version of $T1$ theorem on product spaces. Subsequently, Pott and Villarroya \cite{PV} formulated a new type of $T1$ theorem. Recently, a bi-parameter representation of singular integrals in expression of the dyadic shifts was given by Martikainen \cite{M2012}, which extended the famous result of Hyt\"{o}nen \cite{H} for one-parameter case. Moreover, a bi-parameter version of $T1$ theorem in spaces of non-homogeneous type was demonstrated by Hyt\"{o}nen and Martikainen \cite{HM}.

Still more recently, Martikainen \cite{M2014} studied a class of bi-parameter square function defined as follows:

\begin{definition}\label{definition 1.1}$($\cite{M2014}$)$
For any $x=(x_1,x_2) \in \R^n\times\R^m$, the bi-parameter Littlewood-Paley square function $g$ is defined by
\begin{align*}
g(f)(x):= \bigg(\int_0^{\infty}\int_0^{\infty}|\theta_{t_1,t_2} f(x_1,x_2)|^2 \frac{dt_1}{t_1} \frac{dt_2}{t_2} \bigg)^{1/2},
\end{align*}
where
$\theta_{t_1,t_2} f(x_1,x_2) = \iint_{\R^n\times\R^m} s_{t_1,t_2}(x_1,x_2,y_1,y_2)f(y_1,y_2) dy_1 \ dy_2.$
\end{definition}
In \cite{M2014}, the kernels $s_{t_1,t_2}$ were assumed to satisfy a natural size estimate, a H\"{o}lder estimates and two symmetric mixed H\"{o}lder and size estimates, they were also assumed to satisfy certain mixed Carleson and size estimates, mixed Carleson and H\"{o}lder estimates and a bi-parameter Carleson condition. Under these assumptions, the author of \cite{M2014} established the $L^2$ boundedness of $g$.

 Later on, by modifying and adding new ingredients in the assumptions, Cao and Xue \cite{CX} introduced and studied the $L^2$ boundedness of the following bi-parameter Littlewood-Paley $g_\lambda^*$-function:
\begin{definition}\label{definition 1.1}$($\cite{CX}$)$
Let $\lambda_1,\lambda_2 >1$, for any $x=(x_1,x_2) \in \R^{n}\times\R^{m}$, the bi-parameter Littlewood-Paley $g_\lambda^*$-function is defined by
\begin{align*}
g_{\vec{\lambda}}^*(f)(x)
&:= \bigg(\iint_{\R^{m+1}_{+}} \Big(\frac{t_2}{t_2 + |x_2 - y_2|}\Big)^{m \lambda_2}
\iint_{\R^{n+1}_{+}} \Big(\frac{t_1}{t_1 + |x_1 - y_1|}\Big)^{n \lambda_1} \\
&\quad\quad\quad \times |\vartheta_{t_1,t_2} f(y_1,y_2)|^2 \frac{dy_1 dt_1}{t_1^{n+1}} \frac{dy_2 dt_2}{t_2^{m+1}} \bigg)^{1/2},
\end{align*}
where
$\vartheta_{t_1,t_2} f(y_1,y_2) = \iint_{\R^{n}\times\R^{m}} K_{t_1,t_2}(y_1,y_2,z_1,z_2)f(z_1,z_2) dz_1 \ dz_2.$
\end{definition}

Till now, the only known result about $g$-function and $g_{\vec{\lambda}}^*$-functionis is the $L^2$ boundedness. Therefore, it leaves several questions open, such as, it is quite natural to ask if the $L^p$ boundedness and $H^1$ to $L^1$ boundedness are true or not. In this paper, our objects
of investigation are the bi-parameter Littlewood-Paley $g$-function and $g_{\vec{\lambda}}^*$-function.  We are mainly concerned with the boundedness of these two operators on product Hardy spaces (as consequences, the $L^p$ boundedness will be obtained for $1<p<2$). The introduction of product Hardy spaces will be given in the next subsection.
\subsection{Product Hardy space}
The product Hardy space was first introduced and studied by Malliavin and Malliavin \cite{MM} in 1977, and systematically studied by Gundy and Stein \cite{GS} in 1979. Later on, Chang and R. Fefferman \cite{CF2} established the atomic decomposition of $H^p(\R^2_+\times\R^2_+)$, which is more complicated than the classical $H^p(\R^n)$. By using the rectangle atomic decomposition of $H^p(\R^n\times\R^m)$ (Theorem A below) and a geometric covering lemma due to Journ\'{e}\cite{Jour}, R. Fefferman \cite{Fe} established the $(H^p(\R^n\times\R^m), L^p(\R^n\times\R^m))$ boundedness of Journ\'{e}'s product singular integrals. For more work about the boundedness of operators on product Hardy space, one may refer to \cite{Fe1}, \cite{HLLL}, \cite{HY}, \cite{YZ}.\par
In order to state some known results, we need to introduce two more definitions.
\begin{definition}[\textbf{$H^p(\R^n\times\R^m)$ atom}, \cite{CF}]\label{definition 1.2}
Let $a(x_1,x_2)$ be a function supported in an open set $\Omega\subset\R^n\times\R^m$ with finite measure. $a(x_1,x_2)$ is said to be an $H^p(\R^n\times\R^m)$ atom if it satisfies the following condition: \begin{enumerate}
\item[(i)]
 $\norm{a}_2\leq|\Omega|^{1/2-1/p}$.\\

\item[(ii)] $a$ can further be decomposed as $a(x_1, x_2)=\sum_{R\in \mathcal{M}(\Omega)}a_{R}(x_1,x_2)$, where $a_R$ are supported on the double of $R=I\times J$ ($I$ a dyadic cube in $\R^n$, $J$ a dyadic cube in $\R^m$) and $\mathcal{M}(\Omega)$ is the collection of all maximal dyadic rectangles contained in $\Omega$. Moreover,
$
\big\{\sum_{R\in \mathcal{M}(\Omega)}\norm{a_R}_2^2\big\}^{1/2}\leq|\Omega|^{1/2-1/p}.
$
\item[(iii)] $\int_{2I}a_R(x_1,x_2)x_1^{\alpha}dx_1=0$ for all $x_2\in \R^m$, $0\leq|\alpha|\leq N_{p,n}$,\par $\int_{2J}a_R(x_1,x_2)x_2^{\alpha}dx_2=0$ for all $x_1\in \R^n$, $0\leq|\beta|\leq N_{p,m}$,\par
 where $N_{p,n}$ and $N_{p,m}$ is a large integer depending on $p$ and $n$, \end{enumerate}\par
\end{definition}

\begin{definition}[\textbf{$H^p(\R^n\times\R^m)$ rectangle atom}, \cite{CF}]\label{definition 1.3}Let function $a(x_1,x_2)$ be supported on a rectangle $R=I\times J$, where $I$ is a cube in $\R^n$ and $J$ is a cube in $\R^m$, respectively, $a(x_1,x_2)$ is called an $H^p(\R^n\times\R^m)$ rectangle atom provided\par\begin{enumerate}
\item[(i)]$\norm{a}_2\leq|R|^{1/2-1/p}$,\par
\item[(ii)]$\int_Ia(x_1,x_2)x_1^{\alpha}dx_1=0$ for all $x_2\in\R^m$, $0\leq|\alpha|\leq N_{n,p}$,\par
\item[(iii)]$\int_Ja(x_1,x_2)x_2^{\beta}dx_2=0$ for all $x_1\in\R^n$, $0\leq|\beta|\leq N_{m,p}$.\end{enumerate}
\end{definition}
Chang and R. Fefferman \cite{CF} gave the following atomic decomposition of $H^p(\R^n\times\R^m)$.
\begin{thmA}$($\cite{CF}$)$
A distribution $f\in H^p(\R^n\times\R^m)$ if and only if $f=\sum_j\lambda_ja_j$, where $a_j$ are $H^p(\R^n\times\R^m)$ atoms, $\sum_j|\lambda_j|^p<\infty$, and the series converges in the distribution sense. Moreover, $\norm{f}^p_{H^p}$ is equivalent to $\inf\{\text{$\sum_j|\lambda_j|^p$: for all $f=\sum_j\lambda_ja_j$}\}$.
\end{thmA}
Since the support of $H^p(\R^n\times\R^m)$ atom is an open set, the decomposition of $H^p(\R^n\times\R^m)$ cannot be carried out to demonstrate the $H^p$ boundedness of linear or sublinear operators directly as the classical cases. However, it was quite surprising that R. Fefferman \cite{Fe} obtained the following nice result,

\begin{thmB}$($\cite{Fe}$)$
Suppose that $T$ is a bounded linear operator on $L^2(\R^n\times\R^m)$. Suppose further that if $a$ is an $H^p(\R^n\times\R^m)$ rectangle atom $(0<p\leq 1)$ supported on $R$, we have
\begin{align*}
\iint_{\R^n\times\R^m\setminus\widetilde{R}_{\gamma}}|T(a)(x_1,x_2)|^pdx_1dx_2\leq C\gamma^{-\delta},\quad \text {\ \ for all \ }\gamma\geq 2 \hbox{\ and \ some \ fixed \ }\delta>0,
\end{align*}
where $\widetilde{R}_{\gamma}$ denotes the $\gamma$ fold enlargement of $R$. Then $T$ is a bounded operator from $H^p(\R^n\times\R^m)$ to $L^p(\R^n\times\R^m)$.
\end{thmB}

\subsection{Assumptions and Main result}
\ \

\vspace{0.3cm}
In order to state our main results, we first present the following assumptions:\\
\noindent\textbf{Assumption 1 (Standard estimates of $s_{t_1,t_2}$).} The kernel $s_{t_1,t_2}: (\R^n\times\R^m) \times (\R^n\times\R^m) \rightarrow \C$ is assumed to satisfy the following estimates:
\begin{enumerate}
\item [(1)] Size condition :
$$ |s_{t_1,t_2}(x,y)|
\lesssim \frac{t_1^{\alpha}}{(t_1 + |x_1 - y_1|)^{n+\alpha}} \frac{t_2^{\beta}}{(t_2 + |x_2 - y_2|)^{m+\beta}}.$$
\item [(2)] H\"{o}lder condition :
\begin{align*}
|s_{t_1,t_2}(x,y) - s_{t_1,t_2}(x,(y_1,y_2')) &- s_{t_1,t_2}(x,(y_1',y_2)) + s_{t_1,t_2}(x,y')| \\
& \lesssim \frac{|y_1 - y_1'|^{\alpha}}{(t_1 + |x_1 - y_1|)^{n+\alpha}} \frac{|y_2 - y_2'|^{\beta}}{(t_2 + |x_2 - y_2|)^{m+\beta}},
\end{align*}
whenever $|y_1 -y_1'| < t_1/2$ and $|y_2 - y_2'| < t_2/2$.
\item [(3)] Mixed H\"{o}lder and size conditions :
$$ |s_{t_1,t_2}(x,y) - s_{t_1,t_2}(x,(y_1,y_2'))|
\lesssim \frac{t_1^{\alpha}}{(t_1 + |x_1 - y_1|)^{n+\alpha}} \frac{|y_2 - y_2'|^{\beta}}{(t_2 + |x_2 - y_2|)^{m+\beta}},$$
whenever $|y_2 -y_2'| < t_2/2$ and
$$ |s_{t_1,t_2}(x,y) - s_{t_1,t_2}(x,(y_1',y_2))|
\lesssim \frac{|y_1 - y_1'|^{\alpha}}{(t_1 + |x_1 - y_1|)^{n+\alpha}} \frac{t_2^{\beta}}{(t_2 + |x_2 - y_2|)^{m+\beta}},$$
whenever $|y_1 - y_1'| < t_1/2$.
\end{enumerate}
\noindent\textbf{Assumption 2 (Carleson condition $\times$ Standard estimates of $s_{t_1,t_2}$).}
If $I \subset \Rn$ is a cube with side length $\ell(I)$, we define the associated Carleson box by $\widehat{I}=I \times (0,\ell(I))$. We assume the following conditions : For every cube $I \subset \Rn$ and $J \subset \R^m$, there holds that
\begin{enumerate}
\item [(1)] Combinations of Carleson and size conditions :
\begin{align*}
\Big( \iint_{\widehat{I}}  \Big| \int_{I} s_{t_1,t_2}(x_1,x_2,y_1,y_2) dy_1 \Big|^2 \frac{dx_1 dt_1}{t_1} \Big)^{1/2} \lesssim |I|^{1/2}\frac{ t_2^{\beta}}{(t_2 + |x_2 - y_2|)^{m + \beta}}
\end{align*}
and
\begin{align*}
\Big( \iint_{\widehat{J}}  \Big| \int_{J} s_{t_1,t_2}(x_1,x_2,y_1,y_2) dy_2\Big|^2 \frac{dx_2 dt_2}{t_2} \Big)^{1/2} \lesssim\frac{ t_1^{\alpha}}{(t_1 + |x_1 - y_1|)^{m + \alpha}}|J|^{1/2}.
\end{align*}
\item [(2)] Combinations of Carleson and H\"{o}lder conditions :
\begin{align*}
\Big( \iint_{\widehat{I}} \Big| \int_{I} [ s_{t_1,t_2}(x_1,x_2,y_1,y_2) - s_{t_1,t_2}(x_1,x_2,&y_1, y_2') ] dy_1 \Big|^2\frac{dx_1 dt_1}{t_1}\Big)^{1/2} \\
&\lesssim |I|^{1/2} \frac{|y_2 - y_2'|^{\beta}}{(t_2 + |x_2 - y_2|)^{m + \beta}},
\end{align*}
\noindent whenever $|y_2 - y_2'| < t_2/2$. And
\begin{align*}
\Big( \iint_{\widehat{J}} \Big| \int_{J} [ s_{t_1,t_2}(x_1,x_2,y_1,y_2) - s_{t_1,t_2}(x_1,x_2,&y_1', y_2) ] dy_2 \Big|^2\frac{dx_2 dt_2}{t_2} \Big)^{1/2} \\
&\lesssim |J|^{1/2} \frac{|y_1 - y_1'|^{\alpha}}{(t_1 + |x_1 - y_1|)^{m + \alpha}},
\end{align*}
whenever $|y_1 - y_1'| < t_1/2$.
\end{enumerate}
\noindent\textbf{Assumption 2$'$ (Strong Carleson condition $\times$ Standard estimates of $s_{t_1,t_2}$).}
If $I \subset \Rn$ is a cube with side length $\ell(I)$, we define the associated Carleson box by $\widehat{I}=I \times (0,\ell(I))$. We assume the following conditions : For every cube $I \subset \Rn$ and $J \subset \R^m$, there holds that
\begin{enumerate}
\item [(1)] Combinations of Carleson and size conditions :
\begin{align*}
\Big( \iint_{\widehat{I}}   \int_{I} |s_{t_1,t_2}(x_1,x_2,y_1,y_2)|^2 dy_1  \frac{dx_1 dt_1}{t_1} \Big)^{1/2} \lesssim \frac{ t_2^{\beta}}{(t_2 + |x_2 - y_2|)^{m + \beta}}
\end{align*}
and
\begin{align*}
\Big( \iint_{\widehat{J}}   \int_{J} |s_{t_1,t_2}(x_1,x_2,y_1,y_2)|^2 dy_2  \frac{dx_2 dt_2}{t_2} \Big)^{1/2} \lesssim\frac{ t_1^{\beta}}{(t_1 + |x_1 - y_1|)^{m + \alpha}}.
\end{align*}
\item [(2)] Combinations of Carleson and H\"{o}lder conditions :
\begin{align*}
\Big( \iint_{\widehat{I}} \int_{I} | s_{t_1,t_2}(x_1,x_2,y_1,y_2) - s_{t_1,t_2}(x_1,x_2,&y_1, y_2') |^2 dy_1 \frac{dx_1 dt_1}{t_1} \Big)^{1/2} \\
&\lesssim \frac{|y_2 - y_2'|^{\beta}}{(t_2 + |x_2 - y_2|)^{m + \beta}},
\end{align*}
\noindent whenever $|y_2 - y_2'| < t_2/2$. And
\begin{align*}
\Big( \iint_{\widehat{J}} \int_{J} | s_{t_1,t_2}(x_1,x_2,y_1,y_2) - s_{t_1,t_2}(x_1,x_2,&y_1', y_2) |^2 dy_2 \frac{dx_2 dt_2}{t_2} \Big)^{1/2} \\
&\lesssim \frac{|y_1 - y_1'|^{\alpha}}{(t_1 + |x_1 - y_1|)^{m + \alpha}},
\end{align*}
whenever $|y_1 - y_1'| < t_1/2$.
\end{enumerate}
\noindent\textbf{Assumption 3 (Bi-parameter Carleson condition of $s_{t_1,t_2}$).}
Let $\mathcal{D}=\mathcal{D}_n \times \mathcal{D}_m$, where $\mathcal{D}_n$ is a dyadic grid in $\Rn$ and $\mathcal{D}_m$ is a dyadic grid in $\R^m$. For $I \in \mathcal{D}_n$, let $W_I = I \times (\ell(I)/2, \ell(I))$ be the associated Whitney region. Denote
\begin{align*}
C_{I J}^{\mathcal{D}} &= \iint_{W_J} \iint_{W_I}|\theta_{t_1,t_2} \mathbf{1}(y_1, y_2)|^2
\frac{dx_1 dt_1}{t_1} \frac{dx_2 dt_2}{t_2} .
\end{align*}
We assume the following $bi$-$parameter \ Carleson \ condition$: For every $\mathcal{D}=\mathcal{D}_n \times \mathcal{D}_m$ there holds that
\begin{equation}\label{Car-condition}
\sum_{\substack{I \times J \in \mathcal{D} \\ I \times J \subset \Omega}} C_{I J}^{\mathcal{D}} \lesssim |\Omega|
\end{equation}
for all sets $\Omega \subset \R^n\times\R^m$ such that $|\Omega| < \infty$ and such that for every $x \in \Omega$ there exists $I \times J \in \mathcal{D}$ so that $x \in I \times J \subset \Omega$.\par
\vspace{0.3cm}
\noindent\textbf{Assumption 4 (Standard estimates of $K_{t_1,t_2}$).} The kernel $K_{t_1,t_2}: (\R^n\times\R^m) \times (\R^n\times\R^m) \rightarrow \C$ is assumed to satisfy the following estimates:
\begin{enumerate}
\item [(1)] Size condition :
$$ |K_{t_1,t_2}(x,y)|
\lesssim \frac{t_1^{\alpha}}{(t_1 + |x_1 - y_1|)^{n+\alpha}} \frac{t_2^{\beta}}{(t_2 + |x_2 - y_2|)^{m+\beta}}.$$
\item [(2)] H\"{o}lder condition :
\begin{align*}
|K_{t_1,t_2}(x,y) -& K_{t_1,t_2}(x,(y_1,y_2')) - K_{t_1,t_2}(x,(y_1',y_2)) + K_{t_1,t_2}(x,y')| \\
& \lesssim \frac{|y_1 - y_1'|^{\alpha}}{(t_1 + |x_1 - y_1|)^{n+\alpha}} \frac{|y_2 - y_2'|^{\beta}}{(t_2 + |x_2 - y_2|)^{m+\beta}},
\end{align*}
whenever $|y_1 -y_1'| < t_1/2$ and $|y_2 - y_2'| < t_2/2$.
\item [(3)] Mixed H\"{o}lder and size conditions :
$$ |K_{t_1,t_2}(x,y) - K_{t_1,t_2}(x,(y_1,y_2'))|
\lesssim \frac{t_1^{\alpha}}{(t_1 + |x_1 - y_1|)^{n+\alpha}} \frac{|y_2 - y_2'|^{\beta}}{(t_2 + |x_2 - y_2|)^{m+\beta}},$$
whenever $|y_2 -y_2'| < t_2/2$ and
$$ |K_{t_1,t_2}(x,y) - K_{t_1,t_2}(x,(y_1',y_2))|
\lesssim \frac{|y_1 - y_1'|^{\alpha}}{(t_1 + |x_1 - y_1|)^{n+\alpha}} \frac{t_2^{\beta}}{(t_2 + |x_2 - y_2|)^{m+\beta}},$$
whenever $|y_1 - y_1'| < t_1/2$.
\end{enumerate}
\noindent\textbf{Assumption 5 (Carleson condition $\times$ Standard estimates of $K_{t_1,t_2}$).}
If $I \subset \Rn$ is a cube with side length $\ell(I)$, we define the associated Carleson box by $\widehat{I}=I \times (0,\ell(I))$. We assume the following conditions : For every cube $I \subset \Rn$ and $J \subset \R^m$, there holds that
\begin{enumerate}
\item [(1)] Combinations of Carleson and size conditions :
\begin{align*}
\Big( \iint_{\widehat{I}} \int_{\Rn} \Big| \int_{I} K_{t_1,t_2}(x-y,z_1,z_2) dz_1 \Big|^2 \Big(\frac{t_1}{t_1 + |y_1|}\Big)^{n \lambda_1} &\frac{dy_1}{t_1^n} \frac{dx_1 dt_1}{t_1} \Big)^{1/2} \\
&\lesssim |I|^{1/2}\frac{ t_2^{\beta}}{(t_2 + |x_2 - y_2 - z_2|)^{m + \beta}}
\end{align*}
and
\begin{align*}
\Big( \iint_{\widehat{J}} \int_{\R^m} \Big| \int_{J} K_{t_1,t_2}(x-y,z_1,z_2) dz_2 \Big|^2 \Big(\frac{t_2}{t_2 + |y_2|}\Big)^{n \lambda_1} &\frac{dy_2}{t_2^n} \frac{dx_2 dt_2}{t_2} \Big)^{1/2} \\
&\lesssim |J|^{1/2} \frac{t_1^{\alpha}}{(t_1 + |x_1 - y_1 - z_1|)^{n + \alpha}}.
\end{align*}
\item [(2)] Combinations of Carleson and H\"{o}lder conditions :
\begin{align*}
\Big( \iint_{\widehat{I}} \int_{\Rn} \Big| \int_{I} [ K_{t_1,t_2}(x-y,z_1,z_2) - K_{t_1,t_2}(x-y,z_1,z_2')& ] dz_1 \Big|^2 \Big(\frac{t_1}{t_1 + |y_1|}\Big)^{n \lambda_1} \frac{dy_1}{t_1^n} \frac{dx_1 dt_1}{t_1} \Big)^{\frac12} \\
&\lesssim |I|^{1/2} \frac{|z_2 - z_2'|^{\beta}}{(t_2 + |x_2 - y_2 - z_2|)^{m + \beta}},
\end{align*}
\noindent whenever $|z_2 - z_2'| < t_2/2$. And
\begin{align*}
\Big( \iint_{\widehat{J}} \int_{\R^m} \Big| \int_{J} [ K_{t_1,t_2}(x-y,z_1,z_2) - K_{t_1,t_2}(x-y,z_1',z_2)& ] dz_2 \Big|^2 \Big(\frac{t_2}{t_2 + |y_2|}\Big)^{m \lambda_2} \frac{dy_2}{t_2^n} \frac{dx_2 dt_2}{t_2} \Big)^{\frac12} \\
&\lesssim |J|^{1/2} \frac{|z_1 - z_1'|^{\alpha}}{(t_1 + |x_1 - y_1 - z_1|)^{n + \alpha}},
\end{align*}
whenever $|z_1 - z_1'| < t_1/2$.
\end{enumerate}
\noindent\textbf{Assumption 5$'$ (Strong Carleson condition $\times$ Standard estimates of $K_{t_1,t_2}$).}
If $I \subset \Rn$ is a cube with side length $\ell(I)$, we define the associated Carleson box by $\widehat{I}=I \times (0,\ell(I))$. We assume the following conditions : For every cube $I \subset \Rn$ and $J \subset \R^m$, there holds that
\begin{enumerate}
\item [(1)] Combinations of Carleson and size conditions :
\begin{align*}
\Big( \iint_{\widehat{I}} \int_{\Rn} \int_{I}| K_{t_1,t_2}(x-y,z_1,z_2) |^2  \Big(\frac{t_1}{t_1 + |y_1|}\Big)^{n \lambda_1} &\frac{dz_1dy_1}{t_1^n} \frac{dx_1 dt_1}{t_1} \Big)^{1/2} \\
&\lesssim \frac{ t_2^{\beta}}{(t_2 + |x_2 - y_2 - z_2|)^{m + \beta}}
\end{align*}
and
\begin{align*}
\Big( \iint_{\widehat{J}} \int_{\R^m} \int_{J} |K_{t_1,t_2}(x-y,z_1,z_2)|^2  \Big(\frac{t_2}{t_2 + |y_2|}\Big)^{n \lambda_1} &\frac{dz_2dy_2}{t_2^n} \frac{dx_2 dt_2}{t_2} \Big)^{1/2} \\
&\lesssim  \frac{t_1^{\alpha}}{(t_1 + |x_1 - y_1 - z_1|)^{n + \alpha}}.
\end{align*}
\item [(2)] Combinations of Carleson and H\"{o}lder conditions :
\begin{align*}
\Big( \iint_{\widehat{I}} \int_{\Rn} \int_{I} | K_{t_1,t_2}(x-y,z_1,z_2) - K_{t_1,t_2}(x-y,z_1,&z_2') |^2 \Big(\frac{t_1}{t_1 + |y_1|}\Big)^{n \lambda_1} \frac{dz_1 dy_1}{t_1^n} \frac{dx_1 dt_1}{t_1} \Big)^{1/2} \\
&\lesssim \frac{|z_2 - z_2'|^{\beta}}{(t_2 + |x_2 - y_2 - z_2|)^{m + \beta}},
\end{align*}
\noindent whenever $|z_2 - z_2'| < t_2/2$. And
\begin{align*}
\Big( \iint_{\widehat{J}} \int_{\R^m} \int_{J} | K_{t_1,t_2}(x-y,z_1,z_2) - K_{t_1,t_2}(x-y,&z_1',z_2) |^2  \Big(\frac{t_2}{t_2 + |y_2|}\Big)^{m \lambda_2} \frac{dz_2dy_2}{t_2^n} \frac{dx_2 dt_2}{t_2} \Big)^{1/2} \\
&\lesssim \frac{|z_1 - z_1'|^{\alpha}}{(t_1 + |x_1 - y_1 - z_1|)^{n + \alpha}},
\end{align*}
whenever $|z_1 - z_1'| < t_1/2$.
\end{enumerate}
\noindent\textbf{Assumption 6 (Bi-parameter Carleson condition  of $K_{t_1,t_2}$).}
Let $\mathcal{D}=\mathcal{D}_n \times \mathcal{D}_m$, where $\mathcal{D}_n$ is a dyadic grid in $\Rn$ and $\mathcal{D}_m$ is a dyadic grid in $\R^m$. For $I \in \mathcal{D}_n$, let $W_I = I \times (\ell(I)/2, \ell(I))$ be the associated Whitney region. Denote $n_1=n,n_2=m$ and
\begin{align*}
C_{I J}^{\mathcal{D}} &= \iint_{W_J} \iint_{W_I} \iint_{\R^{n+m}}|\theta_{t_1,t_2} \mathbf{1}(y_1, y_2)|^2
\Big[\prod_{i=1}^2\Big(\frac{t_i}{t_i + |x_i - y_i|}\Big)^{n_i \lambda_i} \bigg] \frac{dy_1dy_2}{t_1^nt_2^m}
\frac{dx_1 dt_1}{t_1} \frac{dx_2 dt_2}{t_2} .
\end{align*}
We assume the following $bi$-$parameter \ Carleson \ condition$: For every $\mathcal{D}=\mathcal{D}_n \times \mathcal{D}_m$ there holds that
\begin{equation}\label{Car-condition}
\sum_{\substack{I \times J \in \mathcal{D} \\ I \times J \subset \Omega}} C_{I J}^{\mathcal{D}} \lesssim |\Omega|
\end{equation}
for all sets $\Omega \subset \R^{n+m}$ such that $|\Omega| < \infty$ and such that for every $x \in \Omega$ there exists $I \times J \in \mathcal{D}$ so that $x \in I \times J \subset \Omega$.

The $L^2(\R^n\times\R^m)$ boundedness of bi-parameter $g$ function is given by Martikainen \cite{M2014}, where the kernel $s_{t_1,t_2}$ satisfies the \textbf{Assumptions} \textbf{1}, \textbf{2}, \textbf{3}. Cao and Xue \cite{CX} establish the $L^2(\R^n\times\R^m)$ boundedness for bi-parameter $g_{\vec{\lambda}}^*$-function, where the kernel $K_{t_1,t_2}$ satisfies the \textbf{Assumptions} \textbf{4}, \textbf{5}, \textbf{6}. It deserves to note that the \textbf{Assumptions 3} is a necessary condition for the $L^2$ boundedness of bi-parameter $g$-function to be held, and \textbf{Assumptions 6} is also necessary for the $L^2$ boundedness of bi-parameter $g_{\vec{\lambda}}^*$-function.
\begin{remark} It is easy to see that Assumptions 2 is a little weaker than Assumptions 2$'$ and the same is true for Assumptions 5 and Assumption 5$'$.\end{remark}

Now we state the main result of this paper.
\begin{theorem}\label{thm1}
Assume that the kernel $s_{t_1,t_2}$ satisfies the \textbf{Assumptions} \textbf{1}, \textbf{2$'$}, \textbf{3}. Then, it holds that
\begin{align}\label{Re}
\big\|g(f) \big\|_{L^1(\R^n\times\R^m)} \lesssim \big\| f \big\|_{H^1(\R^n\times\R^m)},
\end{align}
where the implied constant depends only on the assumptions.
\end{theorem}

\begin{theorem}\label{thm2}
Let $\lambda_1,\lambda_2 >2$, $0 < \alpha \leq n(\lambda_1 -2)/2$ and $0 < \beta \leq m(\lambda_2 -2)/2$. Assume that the kernel $K_{t_1,t_2}$ satisfies the \textbf{Assumptions} \textbf{4}, \textbf{5$'$}, \textbf{6}. Then, it holds that
\begin{align}\label{Reg}
\big\| g_{\vec{\lambda}}^*(f) \big\|_{L^1(\R^n\times\R^m)} \lesssim \big\| f \big\|_{H^1(\R^n\times\R^m)},
\end{align}
where the implied constant depends only on the assumptions.
\end{theorem}

By interpolation, we have the following corollaries:
\begin{corollary}
Assume that the kernel $s_{t_1,t_2}$ satisfies the \textbf{Assumptions} \textbf{1}, \textbf{2$'$}, \textbf{3}. Then $g$ is bounded on $L^p(\R^n\times\R^m)$ for all $1<p<2$.
\end{corollary}

\begin{corollary}
Let $\lambda_1,\lambda_2 >2$, $0 < \alpha \leq n(\lambda_1 -2)/2$ and $0 < \beta \leq m(\lambda_2 -2)/2$. Assume that the kernel $K_{t_1,t_2}$ satisfies the \textbf{Assumptions} \textbf{4}, \textbf{5$'$}, \textbf{6}. Then $g_{\vec{\lambda}}^*$ is bounded on $L^p(\R^n\times\R^m)$ for all $1<p<2$.
\end{corollary}
The organizations of this paper are as follows: Section $\ref{Sec-30}$ and Section $\ref{Sec-300}$ will be devoted to demonstrate Theorem \ref{thm1} and Theorem \ref{thm2}. In section $\ref{Sec-3000}$, we will briefly give the proofs of the Corollaries.
\section{Proof of Theorem 1.1}\label{Sec-30}
In order to demonstrate Theorem \ref{thm1}, we first introduce some notions and a key lemma.

Let $\mathcal{H}_i$ $(i=1,2)$ be the Hilbert space defined by
\begin{align*}
\mathcal{H}_i=\{f: (0,\infty)\times(0,\infty) \rightarrow \mathbb{C} \ \ \ \text{is measurable},\ \  |f|_{\mathcal{H}_i}<\infty\},
\end{align*}
where
\begin{align*}
|f|_{\mathcal{H}_1}=\bigg\{\int_0^{\infty}\int_0^{\infty}|f(t_1,t_2)|^2\frac{dt_1}{t_1}\frac{dt_2}{t_2}\bigg\}^{1/2}
\end{align*}
and
\begin{align*}
|f|_{\mathcal{H}_2}=\bigg\{\int_{\R^{m+1}_+}\int_{\R^{n+1}_+}|f(t_1,t_2,y_1,y_2)|^2\frac{dy_1dt_1}{t_1}\frac{dy_2dt_2}{t_2}\bigg\}^{1/2}.
\end{align*}
Hence, $g$ and $g_{\vec{\lambda}}^*$ can be written by
\begin{align*}
g(f)(x)= |\{\theta_{t_1,t_2} f(x_1,x_2)\}_{t_1,t_2>0}|_{\mathcal{H}_1}:=|T_1|_{\mathcal{H}_1},
\end{align*}
and
\begin{align*}
g_{\vec{\lambda}}^*(f)(x)&= \bigg|\bigg\{\frac{1}{t_1^nt_2^m}\Big(\frac{t_1}{t_1 + |x_1 - y_1|}\Big)^{\frac{n \lambda_1}{2}}\Big(\frac{t_2}{t_2 + |x_2 - y_2|}\Big)^{\frac{m \lambda_2}{2}}\vartheta_{t_1,t_2} f(x_1,x_2)\bigg\}_{t_1,t_2>0}\bigg|_{\mathcal{H}_2}\\
&:=|T_2|_{\mathcal{H}_2}.
\end{align*}
For $p>0$ and $i=1, 2$, we define
\begin{align*}
L_{\mathcal{H}_i}^p(\R^n\times\R^m)=\bigg\{T_i(x_1,x_2):\R^n\times\R^m\rightarrow \mathcal{H}_i\ \ \text{and}\ \ \norm{T_i}_{L^p_{\mathcal{H}_i}(\R^n\times\R^m)}<\infty\bigg\}
\end{align*}
where
$
\norm{T_i}_{L^p_{\mathcal{H}_i}(\R^n\times\R^m)}=
\Big\{\iint_{\R^n\times\R^m}|T_i|^p_{\mathcal{H}_i}dx_1dx_2\Big\}^{1/p}.
$

The following vector-valued version of Theorem B provide a foundation for our proof.
\begin{lemma}\label{lemma}
Let $T_i$ ($i=1,2$) be an $\mathcal{H}_i$-valued linear operator which is bounded from $L^2(\R^n\times\R^m)$ to $L^2_{\mathcal{H}_i}(\R^n\times\R^m)$, and $0<p\leq 1$. Suppose further that there exist constants $C>0$ such that for all $H^p(\R^n\times\R^m)$ rectangle atoms a supported on $R$, it holds that
\begin{align*}
\iint_{\R^n\times\R^m\setminus\widetilde{R}_{\gamma}}|T_i(a)(x_1,x_2)|^p_{\mathcal{H}_i}dx_1dx_2\leq C\gamma^{-\delta}\text {\ \ for all \ }\gamma\geq 2 \hbox{\ and \ some \ fixed \ }\delta>0,
\end{align*}
where $\widetilde{R}_{\gamma}$ denotes the $\gamma$ fold enlargement of $R$. Then there exists a constant $C>0$ such that for all $H^p(\R^n\times\R^m)$-atom $\widetilde{a}$, we have
\begin{align*}
\norm{T_i(\widetilde{a})}_{L_{\mathcal{H}_i}^p(\R^n\times\R^m)}\leq C.
\end{align*}
\end{lemma}
\begin{remark}
The proof of Lemma \ref{lemma} follows from the same idea of R. Fefferman. One only needs to replace $T(a)$ and $L^p(\R^n\times\R^m)$ in the proof given in \cite{Fe} by $T_i(a)$ and $L^p_{\mathcal{H}_i}(\R^n\times\R^m)$, respectively. Here, we omit the proof.
\end{remark}

\begin{proof}[Proof of Theorem~\ref{thm1}]
It suffices to verify that for any $H^1(\R^n\times\R^m)$-atom $\widetilde{a}$, there exists a constant $C>0$, such that
\begin{align}\label{Inq13}
\norm{g(\widetilde{a})}_{L^1(\R^n\times\R^m)}=\norm{T_1(\widetilde{a})}_{L^1_{\mathcal{H}_1}(\R^n\times\R^m)}\leq C.
 \end{align}
 In fact, for any $f\in H^1(\R^n\times\R^m)$, using atomic decomposition (Theorem A), we may split $f$ by $f=\sum_j^{\infty}\lambda_j\widetilde{a}_j$, where $\widetilde{a}_j$ is $H^1(\R^n\times\R^m)$-atom. By the size condition of $s_{t_1,t_2}$ and Fubini theorem, for all $t_1,t_2>0$, we have
 \begin{align*}
 &|\sum_{j=1}^{\infty}\lambda_j\iint_{\R^n\times\R^m}T_1(\widetilde{a}_j)dx_1dx_2|\\
  &\lesssim\sum_{j=1}^{\infty}|\lambda_j|\iint_{\R^n\times\R^m}\iint_{\R^n\times\R^m}
 |s_{t_1,t_2}(x_1,x_2,y_1,y_2)|dx_1dx_2 |\widetilde{a_j}(y_1,y_2)|dy_1dy_2\\
  &\lesssim\sum_{j=1}^{\infty}|\lambda_j|\iint_{\R^n\times\R^m}\iint_{\R^n\times\R^m}
 \frac{t_1^{\alpha}}{(t_1+|x_1-y_1|)^{n+\alpha}}\frac{t_2^{\beta}}{(t_2+|x_2-y_2|)^{n+\beta}}dx_1dx_2\\
  &\quad\quad\quad|\widetilde{a_j}(y_1,y_2)|dy_1dy_2\\
  &\lesssim\sum_{j=1}^{\infty}|\lambda_j|\iint_{\R^n\times\R^m}|\widetilde{a_j}(y_1,y_2)|dy_1dy_2\lesssim\sum_{j=1}^{\infty}|\lambda_j|.
 \end{align*}
Therefore, by Lebesgue dominated convergence theorem, it implies that
 \begin{equation}\label{ae}
 T_1(f)=\sum_{j=1}^{\infty}\lambda_jT_1(\widetilde{a}_j),\ \ a. e.
 \end{equation}
Thus, by the above property (\ref{ae}), inequality (\ref{Inq13}) and Fatou's lemma, we have
 \begin{align}\label{iin9}
 \norm{T_1(f)}_{L^1_{\mathcal{H}_1}(\R^n\times\R^m)}&=\int_{\R^n\times\R^m}|
 \sum_{j=1}^{\infty}\lambda_jT_1(\widetilde{a}_j)(x_1,x_2)|_{\mathcal{H}_1}dx_1dx_2\\ \notag
 &\leq\int_{\R^n\times\R^m}\sum_{j=1}^{\infty}|\lambda|_j|
 T_1(\widetilde{a}_j)(x_1,x_2)|_{\mathcal{H}_1}dx_1dx_2\\ \notag
 &\leq\sum_{j=1}^{\infty}|\lambda|_j\norm{T_1(\widetilde{a}_j)}_{L^1_{\mathcal{H}_1}(\R^n\times\R^m)}\\ \notag
 &\leq\sum_{j=1}^{\infty}|\lambda|_j\lesssim\norm{f}_{H^1(\R^n\times\R^m)}.
 \end{align}
Hence, to prove Theorem \ref{thm1}, it is sufficient to prove inequality (\ref{Inq13}).\par
 Let $a$ be a $H^1(\R^n\times\R^m)$ rectangle atom, supported on a rectangle $R=I\times J$. Note that $g$-function is bounded on $L^2(\R^n\times\R^m)$. That is, $T_1$ is bounded from $L^2(\R^n\times\R^m)$ to $L^2_{\mathcal{H}_1}(\R^n\times\R^m)$. Thus, by Lemma \ref{lemma}, in order to prove inequality $(\ref{Inq13})$, it suffices to verify that 
\begin{align}\label{Re1}
\iint_{\R^n\times\R^m\setminus\widetilde{R}_{\gamma}}|g(a)(x_1,x_2)|dx_1dx_2\lesssim\gamma^{-\delta}, \quad \text {\ \ for all $\gamma\geq 2$}.
\end{align}
Set $\gamma_1,\gamma_2\geq 2$ and $\gamma_1,\gamma_2\sim\gamma$, it's easy to see that
\begin{align*}
\widetilde{R}_{\gamma}&\subset\{(x_1,x_2)\in\R^n\times\R^m: x_1\notin\gamma_1 I,x_2\notin\gamma J\}\cup
\{(x_1,x_2)\in\R^n\times\R^m: x_1\in\gamma_1 I,x_2\notin\gamma J\}\\
&\cup\{(x_1,x_2)\in\R^n\times\R^m: x_1\notin\gamma I,x_2\notin\gamma_2J\}\cup
\{(x_1,x_2)\in\R^n\times\R^m: x_1\notin\gamma I,x_2\in\gamma_2J\}\}\\
&=:E_1\cup E_2\cup E_3\cup E_4.
\end{align*}
By symmetry, to prove (\ref{Re1}), and thus to finish the proof of Theorem \ref{thm1}, we only need to prove that there exists a positive $\delta$, such that
\begin{align}\label{Re2}
\iint_{E_1}|g(a)(x_1,x_2)|dx_1dx_2\lesssim\gamma^{-\delta} \text {\ \ for all $\gamma\geq 2$}
\end{align}
and
\begin{align}\label{Re3}
\iint_{E_2}|g(a)(x_1,x_2)|dx_1dx_2\lesssim\gamma^{-\delta} \text {\ \ for all $\gamma\geq 2$}.
\end{align}

Now let us begin with the proof of inequality
(\ref{Re2}).\\
\noindent
 \bull
\textbf {Proof of (\ref{Re2}).}
By the definition of $g$ and support condition of $a$, we get
\begin{align*}
|g(a)(x)|^2= \int_0^{\infty}\int_0^{\infty}|\iint_{I\times J} s_{t_1,t_2}(x_1,x_2,y_1,y_2)a(y_1,y_2) dy_1 \ dy_2|^2 \frac{dt_1}{t_1} \frac{dt_2}{t_2}.
\end{align*}
 In order to estimate $|g(a)(x)|^2$ (where $x\in E_1$), we splitting the domain of variable $t_1$ and $t_2$ as follows,
\begin{align*}
(t_1,t_2)\in \R^+\times\R^+\subset&\big([|x_1-z_1|,\infty)\times[|x_2-z_2|,\infty)\big)\cup\big([|x_1-z_1|,\infty)\times(0,|x_2-z_2|)\big)\\
&\cup\big((0,|x_1-z_1|)\times[|x_2-z_2|,\infty)\big)\cup\big((0,|x_1-z_1|)\times(0,|x_2-z_2|)\big)\\
&=:F_1\cup F_2\cup F_3\cup F_4.
\end{align*}
Hence, we may write
\begin{align*}
|g(a)(x)|^2&\leq \sum_{i=1}^{4}\int_0^{\infty}\int_0^{\infty}|\iint_{I\times J} \mathbf{1}_{F_i}(t_1,t_2)s_{t_1,t_2}(x_1,x_2,y_1,y_2)a(y_1,y_2) dy_1 \ dy_2|^2 \frac{dt_1}{t_1} \frac{dt_2}{t_2}\\
&=:\sum_{i=1}^4A_i.
\end{align*}
\quad Let $z_1$ be the centre of $I$ and $z_2$ be the centre of $J$, denote $l_I$ and $l_J$ as the sidelength of $I$ and $J$, respectively. Note that, if $x\in E_1$ and $y\in I\times J$, then $|x_1-y_1|\sim|x_1-z_1|$ and $|x_2-y_2|\sim|x_2-z_2|$. We will continue to use the above notions in the rest of the paper, moreover, we always denote $$\Phi^a_{t_1,t_2}(x,y,z): =\frac{|y_1 - z_1|^{\alpha}}{(t_1 + |x_1 - y_1|)^{n+\alpha}} \frac{|y_2 - z_2|^{\beta}}{(t_2 + |x_2 - y_2|)^{m+\beta}}|a(y_1,y_2)|$$
\noindent
\textbf {Estimate for $A_1$.} Since $x\in E_1$ and $(t_1,t_2)\in F_1$, then for any $y_1\in I$, we have $t_1\geq |x_1-z_1|\geq\gamma l_I\geq 2|y_1-z_1|$ and $t_1\gtrsim|x_1-y_1|$. Similarly, for any $y_2\in J$, we have $t_1\geq 2|y_1-z_1|$. The definition of $A_1$, vanishing and size condition of $H^1(\R^n\times\R^m)$ rectangle atom and H\"{o}lder estimate of $s_{t_1,t_2}$ imply that
\begin{align*}
A_1&= \int_0^{\infty}\int_0^{\infty}|\iint_{I\times J} \mathbf{1}_{F_1}(t_1,t_2)(s_{t_1,t_2}(x,y) - s_{t_1,t_2}(x,(y_1,z_2))\\
 &\quad- s_{t_1,t_2}(x,(z_1,y_2)) + s_{t_1,t_2}(x,(z_1,z_2)))a(y_1,y_2) d\vec{y}|^2 \frac{dt_1dt_2}{t_1t_2}\\
&\lesssim \int_{|x_1-z_1|}^{\infty}\int_{|x_2-z_2|}^{\infty}|\iint_{I\times J}\Phi^a_{t_1,t_2}(x,y,z)dy_1dy_2|^2 \frac{dt_1}{t_1} \frac{dt_2}{t_2}\\
&\lesssim l_I^{2\alpha}l_J^{2\beta}\int_{|x_1-z_1|}^{\infty}\int_{|x_2-z_2|}^{\infty}\frac{dt_1}{t_1^{2n+2\alpha+1}} \frac{dt_2}{t_2^{2n+2\beta+1}}
|\iint_{I\times J}|a(y_1,y_2)| dy_1dy_2|^2\\
&\lesssim l_I^{2\alpha}l_J^{2\beta}|x_1-z_1|^{-2n-2\alpha}|x_2-z_2|^{-2m-2\beta}.
\end{align*}
\noindent
\textbf {Estimate for $A_2$.} In this case, $x\in E_1$ and $(t_1,t_2)\in F_2$. Note that
\begin{align*}
F_2&\subset\big([|x_1-z_1|,\infty)\times(0,2|y_2-z_2|)\big)\cup\big([|x_1-z_1|,\infty)\times(2|y_2-z_2|,|x_2-z_2|)\big)\\
&=:F_{2,1}\cup F_{2,2}.
\end{align*}
Then, one may deduce that
\begin{align*}
A_2&\leq\sum_{i=1}^2\int_0^{\infty}\int_0^{\infty}|\iint_{I\times J} \mathbf{1}_{F_{2,i}}(t_1,t_2)s_{t_1,t_2}(x_1,x_2,y_1,y_2)a(y_1,y_2) dy_1 \ dy_2|^2 \frac{dt_1}{t_1} \frac{dt_2}{t_2}\\
&=:A_{2,1}+A_{2,2}.
\end{align*}
First, we consider the contribution of $A_{2,1}$.
Note that if $x\in E_1$ and $(t_1,t_2)\in F_{2,1}$, then for any $y_1\in I$, it holds that $t_1\geq 2|y_1-z_1|$ and $t_1\gtrsim|x_1-y_1|\sim|x_1-z_1|$. In addition, $t_2<2|y_2-z_2|<2l_J<|x_2-y_2|\sim|x_2-z_2|$. By the vanishing and size condition of $H^1(\R^n\times\R^m)$ rectangle atom and mixed H\"{o}lder and size conditions of $s_{t_1,t_2}$, we obtain
\begin{align*}
A_{2,1}&= \int_0^{\infty}\int_0^{\infty}|\iint_{I\times J} \mathbf{1}_{F_{2,1}}(t_1,t_2)(s_{t_1,t_2}(x,y)
- s_{t_1,t_2}(x,(z_1,y_2)))a(y_1,y_2)  d\vec{y}|^2 \frac{dt_1}{t_1} \frac{dt_2}{t_2}\\
&\lesssim \int_{|x_1-z_1|}^{\infty}\int_{0}^{2|y_2-z_2|}[\iint_{I\times J}\frac{|y_1 - z_1|^{\alpha}}{(t_1 + |x_1 - y_1|)^{n+\alpha}} \frac{t_2^{\beta}}{(t_2 + |x_2 - y_2|)^{m+\beta}}|a(y_1,y_2)|  d\vec{y}]^2 \frac{dt_1t_2}{t_1t_2}\\
&\lesssim l_I^{2\alpha}|x_2-z_2|^{-2m-2\beta}\int_{|x_1-z_1|}^{\infty}\int_{0}^{2|y_2-z_2|}\frac{t_2^{\beta-1}}{t_1^{2n+2\alpha+1}} dt_1dt_2
|\iint_{I\times J}|a(y_1,y_2)| dy_1dy_2|^2\\
&\lesssim l_I^{2\alpha}l_J^{2\beta}|x_1-z_1|^{-2n-2\alpha}|x_2-z_2|^{-2m-2\beta}.
\end{align*}
Nextly, we consider $A_{2,2}$.\par
Since $x\in E_1$ and $(t_1,t_2)\in F_{2,2}$, then for any $y_1\in I$, we have $t_1\geq |x_1-z_1|\geq\gamma l_I\geq 2|y_1-z_1|$ and $t_1\gtrsim|x_1-y_1|$. Similarly, for any $y_2\in J$ satisfies $|x_2-y_2|\geq t_1\geq 2|y_1-z_1|$. Set $\varepsilon<\min\{\alpha,\beta\}$. The vanishing and size condition of $H^1(\R^n\times\R^m)$ rectangle atom and H\"{o}lder estimate of $s_{t_1,t_2}$ yield that
\begin{align*}
A_{2,2}&= \int_0^{\infty}\int_0^{\infty}|\iint_{I\times J} \mathbf{1}_{F_{2,2}}(t_1,t_2)(s_{t_1,t_2}(x,y) - s_{t_1,t_2}(x,(y_1,z_2))\\
 &\quad- s_{t_1,t_2}(x,(z_1,y_2)) + s_{t_1,t_2}(x,(z_1,z_2)))a(y_1,y_2) dy_1 \ dy_2|^2 \frac{dt_1}{t_1} \frac{dt_2}{t_2}\\
&\lesssim \int_{|x_1-z_1|}^{\infty}\int_{2|y_2-z_2|}^{|x_2-y_2|}|\iint_{I\times J}\Phi^a_{t_1,t_2}(x,y,z) dy_1dy_2|^2 \frac{dt_1}{t_1} \frac{dt_2}{t_2}\\
&\lesssim l_I^{2\alpha}l_J^{2\beta}\big(\int_{|x_1-z_1|}^{\infty}\frac{dt_1}{t_1^{2n+2\alpha+1}}\big)
\int_{2|y_2-z_2|}^{|x_2-y_2|}|\iint_{I\times J}\frac{|x_2 - y_2|^{\varepsilon}}{t_2^{\varepsilon}|x_2 - y_2|^{m+\beta}}|a(y_1,y_2)| dy_1dy_2|^2\frac{dt_2}{t_2}\\
&\lesssim l_I^{2\alpha}l_J^{2\beta-2\varepsilon}|x_1-z_1|^{-2n-2\alpha}|x_2-z_2|^{-2m-2\beta+2\varepsilon}.
\end{align*}
Thus, it yields that
\begin{align*}
A_2&\lesssim A_{2,1}+A_{2,2}\\
&\lesssim l_I^{2\alpha}l_J^{2\beta}|x_1-z_1|^{-2n-2\alpha}|x_2-z_2|^{-2m-2\beta}+ l_I^{2\alpha}l_J^{2\beta-2\varepsilon}|x_1-z_1|^{-2n-2\alpha}|x_2-z_2|^{-2m-2\beta+2\varepsilon}.
\end{align*}
\textbf {Estimate for $A_3$.} Since $A_3$ is symmetric with $A_2$, we may obtain that
\[A_3\lesssim l_I^{2\alpha}l_J^{2\beta}|x_1-z_1|^{-2n-2\alpha}|x_2-z_2|^{-2m-2\beta}+ l_I^{2\alpha-2\varepsilon}l_J^{2\beta}|x_1-z_1|^{-2n-2\alpha+2\varepsilon}|x_2-z_2|^{-2m-2\beta}.\]
\textbf {Estimate for $A_4$.} Recall that $F_4=(0,|x_1-z_1|)\times(0,|x_2-z_2|)$. In order to estimate $A_4$, we further split $F_4$ as follows,
\begin{align*}
F_4&\subset\big((2|y_1-z_1|,|x_1-z_1|)\times(2|y_1-z_1|,|x_2-z_2|)\big)\cup\big((0,2|y_1-z_1|)\times(0,2|y_2-z_2|)\big)\\
&\quad\cup\big((2|y_1-z_1|,|x_1-z_1|)\times(0,2|y_2-z_2|)\big)\cup\big((0,2|y_1-z_1|)\times(2|y_1-z_1|,|x_1-z_1|)\big)\\
&\quad =:F_{4,1}\cup F_{4,2}\cup F_{4,3}\cup F_{4,4}.
\end{align*}
Therefore, we have
\begin{align*}
A_4&\leq \sum_{i=1}^{4}\int_0^{\infty}\int_0^{\infty}|\iint_{I\times J} \mathbf{1}_{F_{4,i}}(t_1,t_2)s_{t_1,t_2}(x_1,x_2,y_1,y_2)a(y_1,y_2) dy_1dy_2|^2 \frac{dt_1}{t_1} \frac{dt_2}{t_2}\\
&=:\sum_{i=1}^4A_{4,i}.
\end{align*}
Now, we begin with the estimate of $A_{4,1}$.\par
Note that if $x\in E_1$ and $(t_1,t_2)\in F_{4,1}$, then for any $y_1\in I$ and $y_2\in J$, we have $|x_1-z_1|\geq t_1\geq 2|y_1-z_1|$ and $|x_2-z_2|\geq t_2\geq 2|y_2-z_2|$. Thus, by the vanishing and size condition of $H^1(\R^n\times\R^m)$ rectangle atom and H\"{o}lder conditions of $s_{t_1,t_2}$, we get
\begin{align*}
A_{4,1}&= \int_0^{\infty}\int_0^{\infty}|\iint_{I\times J} \mathbf{1}_{F_{4,1}}(t_1,t_2)(s_{t_1,t_2}(x,y) - s_{t_1,t_2}(x,(y_1,z_2))\\
 &\quad- s_{t_1,t_2}(x,(z_1,y_2)) + s_{t_1,t_2}(x,(z_1,z_2)))a(y_1,y_2) dy_1dy_2|^2 \frac{dt_1dt_2}{t_1t_2} \\
&\lesssim \int_{2|y_1-z_1|}^{|x_1-y_1|}\int_{2|y_2-z_2|}^{|x_2-y_2|}|\iint_{I\times J}\Phi^a_{t_1,t_2}(x,y,z) dy_1dy_2|^2 \frac{dt_1}{t_1} \frac{dt_2}{t_2}\\
&\lesssim l_I^{2\alpha}l_J^{2\beta}\int_{2|y_1-z_1|}^{|x_1-y_1|}\int_{2|y_2-z_2|}^{|x_2-y_2|}|\iint_{I\times J}\frac{|x_1 - y_1|^{\varepsilon}}{t_1^{\varepsilon}|x_1 - y_1|^{m+\beta}}\frac{|x_2 - y_2|^{\varepsilon}}{t_2^{\varepsilon}|x_2 - y_2|^{m+\beta}}\\&\quad\times|a(y_1,y_2)| dy_1dy_2|^2\frac{dt_1}{t_1}\frac{dt_2}{t_2}\\
&\lesssim l_I^{2\alpha-2\varepsilon}l_J^{2\beta-2\varepsilon}|x_1-z_1|^{-2n-2\alpha+2\varepsilon}|x_2-z_2|^{-2m-2\beta+2\varepsilon}.
\end{align*}
\quad The vanishing and size conditions of $H^1(\R^n\times\R^m)$ rectangle atom and size estimate of $s_{t_1,t_2}$ yield that
\begin{align*}
A_{4,2}&= \int_0^{\infty}\int_0^{\infty}|\iint_{I\times J} \mathbf{1}_{F_{4,2}}(t_1,t_2)s_{t_1,t_2}(x,y)a(y_1,y_2) dy_1dy_2|^2 \frac{dt_1t_2}{t_1t_2} \\
&\lesssim \int_{0}^{2|y_1-z_1|}\int_{0}^{2|y_2-z_2|}|\iint_{I\times J}\frac{t_1^{\alpha}}{(t_1 + |x_1 - y_1|)^{n+\alpha}} \frac{t_2^{\beta}}{(t_2 + |x_2 - y_2|)^{m+\beta}}\\&\quad \times |a(y_1,y_2)| dy_1dy_2|^2 \frac{dt_1}{t_1} \frac{dt_2}{t_2}\\
&\lesssim |x_1-z_1|^{-2n-2\alpha}|x_2-z_2|^{-2m-2\beta}\int_{2|y_1-z_1|}^{|x_1-y_1|}\int_{2|y_2-z_2|}^{|x_2-y_2|}t_1^{-1-2\alpha}t_2^{-1-2\beta}dt_1dt_2\\
&\lesssim l_I^{2\alpha}l_J^{2\beta}|x_1-z_1|^{-2n-2\alpha}|x_2-z_2|^{-2m-2\beta}.
\end{align*}
Now, we are in the position to estimate $A_{4,3}$.\par
Note that if $x\in E_1$ and $(t_1,t_2)\in F_{4,3}$, then for any $y_1\in I$, it holds that $|x_1-z_1|\geq t_1\geq 2|y_1-z_1|$. By the vanishing and size conditions of $H^1(\R^n\times\R^m)$ rectangle atom and H\"{o}lder conditions of $s_{t_1,t_2}$, we get
\begin{align*}
A_{4,3}&= \int_0^{\infty}\int_0^{\infty}|\iint_{I\times J} \mathbf{1}_{F_{4,3}}(t_1,t_2)(s_{t_1,t_2}(x,y)
- s_{t_1,t_2}(x,(z_1,y_2)))a(y_1,y_2) dy_1dy_2|^2\frac{dt_1t_2}{t_1t_2}\\
&\lesssim \int_{2|y_1-z_1|}^{|x_1-y_1|}\int_{0}^{2|y_2-z_2|}|\iint_{I\times J}\frac{|y_1 - z_1|^{\alpha}}{(t_1 + |x_1 - y_1|)^{n+\alpha}} \frac{t_2^{\beta}}{(t_2 + |x_2 - y_2|)^{m+\beta}}\\&\quad\times|a(y_1,y_2)| dy_1dy_2|^2 \frac{dt_1}{t_1} \frac{dt_2}{t_2}\\
&\lesssim l_I^{2\alpha}l_J^{2\beta}|x_2-z_2|^{-2n-2\beta}\int_{2|y_1-z_1|}^{|x_1-y_1|}|\iint_{I\times J}\frac{|x_1 - y_1|^{\varepsilon}}{t_1^{\varepsilon}|x_1 - y_1|^{n+\alpha}}|a(y_1,y_2)| dy_1dy_2|^2 \frac{dt_1}{t_1} \\
&\lesssim l_I^{2\alpha-2\varepsilon}l_J^{2\beta}|x_1-z_1|^{-2n-2\alpha+2\varepsilon}|x_2-z_2|^{-2m-2\beta}.
\end{align*}
By the fact that $A_{4,4}$ is symmetry with $A_{4,3}$, we may obtain
\[A_{4,4}\lesssim l_I^{2\alpha}l_J^{2\beta-2\varepsilon}|x_1-z_1|^{-2n-2\alpha}|x_2-z_2|^{-2m-2\beta+2\varepsilon}.\]
Therefore, we conclude that
\begin{align*}
|g(a)(x)|^2\leq \sum_{i=1}^4A_i&\lesssim l_I^{2\alpha-2\varepsilon}l_J^{2\beta-2\varepsilon}|x_1-z_1|^{-2n-2\alpha+2\varepsilon}|x_2-z_2|^{-2m-2\beta+2\varepsilon}\\
&\quad+l_I^{2\alpha}l_J^{2\beta}|x_1-z_1|^{-2n-2\alpha}|x_2-z_2|^{-2m-2\beta}\\
&\quad+l_I^{2\alpha-2\varepsilon}l_J^{2\beta}|x_1-z_1|^{-2n-2\alpha+2\varepsilon}|x_2-z_2|^{-2m-2\beta}\\
&\quad+l_I^{2\alpha}l_J^{2\beta-2\varepsilon}|x_1-z_1|^{-2n-2\alpha}|x_2-z_2|^{-2m-2\beta+2\varepsilon}.
\end{align*}
Recall that $\varepsilon<\min\{\alpha,\beta\}$ and $\gamma_1\sim\gamma$. The above estimate leads to
\begin{align*}
\iint_{E_1}|g(a)(x_1,x_2)|dx_1dx_2\lesssim\gamma^{-\min\{\alpha,\beta\}} .
\end{align*}
\noindent
 \bull
\textbf {Proof of (\ref{Re3}).}
Recall that $E_2=\{(x_1,x_2)\in\R^n\times\R^m: x_1\in\gamma_1 I,x_2\notin\gamma J\}$, we have
\begin{align*}
E_2&\subset\{(x_1,x_2)\in\R^n\times\R^m: x_1\in 2I,x_2\notin\gamma J\}\\
&\quad\cup\{(x_1,x_2)\in\R^n\times\R^m: x_1\in\gamma_1 I\backslash 2I,x_2\notin\gamma J\}=:E_{2,1}\cup E_{2,2}.
\end{align*}
Therefore, it holds that
\begin{align*}
\iint_{E_2}|g(a)(x_1,x_2)|dx_1dx_2\lesssim\iint_{E_{2,1}}|g(a)(x_1,x_2)|dx_1dx_2+\iint_{E_{2,2}}|g(a)(x_1,x_2)|dx_1dx_2 .
\end{align*}
Since $E_{2,2}\subset\{(x_1,x_2)\in\R^n\times\R^m: x_1\notin2I,x_2\notin\gamma J\}$, taking $\gamma_1=2$ and repeating the proof of (\ref{Re2}), we may get
\begin{align*}
\iint_{E_{2,2}}|g(a)(x_1,x_2)|dx_1dx_2\lesssim\gamma^{\min\{\alpha,\beta\}} .
\end{align*}
Hence, to complete the proof of Theorem \ref{Re}, we only need to show that there exists a $\delta>0$ such that
\begin{align}\label{R5}
\iint_{E_{2,1}}|g(a)(x_1,x_2)|dx_1dx_2\lesssim\gamma^{-\delta} .
\end{align}
To do this, similar to the analysis in the proof of (\ref{Re2}), we split the domain of variable $t_1$ and $t_2$ as follows,
\begin{align*}
(t_1,t_2)\in \R^+\times\R^+\subset&\big((0,2l_I]\times[|x_2-z_2|,\infty)\big)\cup\big((0,2l_I]\times(2|y_2-z_2|,|x_2-z_2|)\big)\\
&\cup\big((0,2l_I]\times[0,2|y_2-z_2|)\big)\cup\big((2l_I,\infty)\times[|x_2-z_2|,\infty)\big)\\
&\cup\big((2l_I,\infty)\times(2|y_2-z_2|,|x_2-z_2|)\big)\cup\big((2l_I,\infty)\times[0,2|y_2-z_2|)\big)\\
&=:G_1\cup G_2\cup G_3\cup G_4\cup G_5\cup G_6.
\end{align*}
Thus, we may obtain
\begin{align*}
\iint_{E_{2,1}}|g(a)(x_1,x_2)|dx_1dx_2
&\leq \sum_{i=1}^{6}\iint_{E_2}\big(\int_0^{\infty}\int_0^{\infty}|\iint_{I\times J} \mathbf{1}_{G_i}(t_1,t_2)s_{t_1,t_2}(x_1,x_2,y_1,y_2)\\&\quad\times a(y_1,y_2) dy_1 \ dy_2|^2 \frac{dt_1}{t_1} \frac{dt_2}{t_2}\big)^{1/2}dx_1dx_2=:\sum_{i=1}^6B_i.
\end{align*}
Now, let us begin with the estimate of $B_1$.\\
\noindent
\textbf {Estimate for $B_1$.} Since $x\in E_2$ and $(t_1,t_2)\in G_1$, then for any $y_2\in J$, it satisfies that $t_2\geq 2|y_2-z_2|$. The support, vanishing and size condition of $H^1(\R^n\times\R^m)$ rectangle atom, H\"{o}lder inequality and the combinations of Carleson and H\"{o}lder conditions of $s_{t_1,t_2}$ imply that
\begin{align}\label{B-1}
B_1&= \iint_{2I\times(\gamma J)^c}\bigg(\int_0^{\infty}\int_0^{\infty}|\iint_{2I\times J} \mathbf{1}_{G_1}(t_1,t_2)(s_{t_1,t_2}(x,y)
- s_{t_1,t_2}(x,(y_1,z_2)))\\ \notag
&\quad\times a(y_1,y_2) dy_1dy_2|^2 \frac{dt_1}{t_1} \frac{dt_2}{t_2}\bigg)^{1/2}dx_1dx_2\\ \notag
&\lesssim \norm{a}_{L^2}\iint_{2I\times(\gamma J)^c}\bigg(\int_0^{2l_I}\int_{|x_2-z_2|}^{\infty}\iint_{2I\times J} |s_{t_1,t_2}(x,y)
- s_{t_1,t_2}(x,(y_1,z_2))|^2\\ \notag
&\quad\times dy_1dy_2 \frac{dt_1}{t_1} \frac{dt_2}{t_2}\bigg)^{1/2}dx_1dx_2\\ \notag
&\lesssim \norm{a}_{L^2}(2I)^{1/2}\int_{(\gamma J)^c}\bigg(\int_{|x_2-z_2|}^{\infty}\int_J\big(\iint_{\widehat{2I}}\int_{2I} |s_{t_1,t_2}(x,y)
- s_{t_1,t_2}(x,(y_1,z_2))|^2\\ \notag
&\quad\times dy_1 \frac{dx_1 dt_1}{t_1}\big)dy_2\frac{dt_2}{t_2}\bigg)^{1/2}dx_2\\ \notag
&\lesssim \norm{a}_{L^2}I^{1/2}\int_{(\gamma J)^c}\bigg(\int_{|x_2-z_2|}^{\infty}\int_J\frac{|y_2 - z_2|^{2\beta}}{(t_2 + |x_2 - y_2|)^{2m + 2\beta}}dy_2\frac{dt_2}{t_2}\bigg)^{1/2}dx_2\\ \notag
&\lesssim \norm{a}_{L^2}I^{1/2}J^{1/2}l_J^{\beta}\int_{(\gamma J)^c}|x_2-z_2|^{-m-\beta}dx_2\lesssim\gamma^{-\beta}. \notag
\end{align}
Now, we consider the estimate of $B_{2}$.\par
\noindent
\textbf {Estimate for $B_2$.} If $x\in E_2$ and $(t_1,t_2)\in G_1$, then for any $y_2\in J$, it holds that $|x_2-z_2|\geq t_2\geq 2|y_2-z_2|$. Recall $\varepsilon\leq\min\{\alpha,\beta\}$. Similar to inequality (\ref{B-1}), by the support, vanishing and size condition of $H^1(\R^n\times\R^m)$ rectangle atom, H\"{o}lder inequality and the combinations of Carleson and H\"{o}lder conditions of $s_{t_1,t_2}$, we may obtain
\begin{align*}
B_2&= \iint_{2I\times(\gamma J)^c}\bigg(\int_0^{\infty}\int_0^{\infty}|\iint_{2I\times J} \mathbf{1}_{G_2}(t_1,t_2)(s_{t_1,t_2}(x,y)
- s_{t_1,t_2}(x,(y_1,z_2)))\\
&\quad\quad \times a(y_1,y_2) dy_1dy_2|^2 \frac{dt_1}{t_1} \frac{dt_2}{t_2}\bigg)^{1/2}dx_1dx_2\\
&\lesssim \norm{a}_{L^2}I^{1/2}\int_{(\gamma J)^c}\bigg(\int_{2|y_2-z_2|}^{|x_2-z_2|}\int_J\frac{|y_2 - z_2|^{2\beta}}{(t_2 + |x_2 - y_2|)^{2m + 2\beta}}dy_2\frac{dt_2}{t_2}\bigg)^{1/2}dx_2\\
&\lesssim \norm{a}_{L^2}I^{1/2}l_J^{\beta}\int_{(\gamma J)^c}\bigg(\int_{2|y_2-z_2|}^{|x_2-z_2|}\int_J\frac{|x_2 - z_2|^{\varepsilon}}{t_2^{\varepsilon}|x_2 - z_2|^{2m + 2\beta}}dy_2\frac{dt_2}{t_2}\bigg)^{1/2}dx_2\\
&\lesssim \norm{a}_{L^2}I^{1/2}J^{1/2}l_J^{\beta-\varepsilon/2}\int_{(\gamma J)^c}|x_2-z_2|^{-m-\beta+\varepsilon/2}dx_2\lesssim\gamma^{-\beta/2}.
\end{align*}
\noindent
\textbf {Estimate for $B_3$.} The support, vanishing and size condition of $H^1(\R^n\times\R^m)$ rectangle atom, H\"{o}lder inequality and the Combinations of Carleson and size conditions of $s_{t_1,t_2}$ yield that
\begin{align*}
B_3&= \iint_{2I\times(\gamma J)^c}\bigg(\int_0^{\infty}\int_0^{\infty}|\iint_{2I\times J} \mathbf{1}_{G_3}(t_1,t_2)s_{t_1,t_2}(x,y)\\&
\quad\quad \times(y_1,y_2) dy_1dy_2|^2 \frac{dt_1dt_2}{t_1t_2} \bigg)^{1/2}dx_1dx_2\\
&\lesssim \norm{a}_{L^2}(2I)^{1/2}\int_{(\gamma J)^c}\bigg(\int_{0}^{2|y_2-z_2|}\int_J\big(\iint_{\widehat{2I}}\int_{2I}|s_{t_1,t_2}(x_1,x_2,y_1,y_2)|dy_1\frac{dx_1dt_1}{t_1}\big)\\&\quad\quad \times dy_2\frac{dt_2}{t_2}\bigg)^{1/2}dx_2\\
&\lesssim \norm{a}_{L^2}I^{1/2}\int_{(\gamma J)^c}\bigg(\int_{0}^{2|y_2-z_2|}\int_J\frac{t_2^{2\beta}}{(t_2 + |x_2 - y_2|)^{2m + 2\beta}}dy_2\frac{dt_2}{t_2}\bigg)^{1/2}dx_2\\
&\lesssim \norm{a}_{L^2}I^{1/2}J^{1/2}l_J^{\beta}\int_{(\gamma J)^c}|x_2-z_2|^{-m-\beta}dx_2\lesssim\gamma^{-\beta}.
\end{align*}
\noindent
\textbf {Estimate for $B_4$.} Since $x\in E_2$ and $(t_1,t_2)\in G_4$, then for any $y_1\in I$, it satisfies that $t_1\geq 2|y_1-z_1|$ and $t_1>|x_1-y_1|$. For any $y_2\in J$, it satisfies that $t_2\geq|x_2-z_2|\geq 2|y_2-z_2|$. The vanishing and size condition of $H^1(\R^n\times\R^m)$ rectangle atom and H\"{o}lder estimate of $s_{t_1,t_2}$ yield that
\begin{align}\label{B4}
&\int_0^{\infty}\int_0^{\infty}|\iint_{I\times J} \mathbf{1}_{G_4}(t_1,t_2)s_{t_1,t_2}(x,y))a(y_1,y_2) dy_1 \ dy_2|^2 \frac{dt_1}{t_1} \frac{dt_2}{t_2}\\ \notag
&=\int_0^{\infty}\int_0^{\infty}|\iint_{I\times J} \mathbf{1}_{G_4}(t_1,t_2)(s_{t_1,t_2}(x,y) - s_{t_1,t_2}(x,(y_1,z_2))\\ \notag
 &\quad- s_{t_1,t_2}(x,(z_1,y_2)) + s_{t_1,t_2}(x,(z_1,z_2)))a(y_1,y_2) dy_1 \ dy_2|^2 \frac{dt_1}{t_1} \frac{dt_2}{t_2}\end{align}\begin{align*}
{}&\lesssim \int_{2l_I}^{\infty}\int_{|x_2-z_2|}^{\infty}|\iint_{I\times J}\frac{|y_1 - z_1|^{\alpha}}{(t_1 + |x_1 - y_1|)^{n+\alpha}} \frac{|y_2 - z_2|^{\beta}}{(t_2 + |x_2 - y_2|)^{m+\beta}}\\ \notag &\quad \times |a(y_1,y_2)| dy_1dy_2|^2 \frac{dt_1}{t_1} \frac{dt_2}{t_2}\\ \notag
&\lesssim l_I^{2\alpha}l_J^{2\beta}\int_{2l_I}^{\infty}\int_{|x_2-z_2|}^{\infty}\frac{dt_1}{t_1^{2n+2\alpha+1}} \frac{dt_2}{t_2^{2n+2\beta+1}}
|\iint_{I\times J}|a(y_1,y_2)| dy_1dy_2|^2\\ \notag
&\lesssim l_I^{-2n}l_J^{2\beta}|x_2-z_2|^{-2m-2\beta}. \notag
\end{align*}
Putting the above estimates into the definition of $B_4$, it holds that
\begin{align*}
B_4\lesssim l_I^{-n}l_J^{\beta}\int_{2I}\int_{(\gamma J)^c}|x_2-z_2|^{-m-\beta}dx_1dx_2\lesssim\gamma^{-\beta}.
\end{align*}\\
\noindent
\textbf {Estimate for $B_5$.} Note that $x\in E_2$ and $(t_1,t_2)\in G_5$, then for any $y_1\in I$, it satisfies that $t_1\geq 2|y_1-z_1|$ and $t_1>|x_1-y_1|$. For any $y_2\in J$, it satisfies $|x_2-z_2|\geq t_2\geq 2|y_2-z_2|$. Similar to inequality (\ref{B4}), by the vanishing and size condition of $H^1(\R^n\times\R^m)$ rectangle atom and H\"{o}lder estimate of $s_{t_1,t_2}$, we may get
\begin{align*}
&\int_0^{\infty}\int_0^{\infty}|\iint_{I\times J} \mathbf{1}_{G_5}(t_1,t_2)s_{t_1,t_2}(x,y))a(y_1,y_2) dy_1 \ dy_2|^2 \frac{dt_1}{t_1} \frac{dt_2}{t_2} \\
&\lesssim \int_{2l_I}^{\infty}\int_{2|y_2-z_2|}^{|x_2-z_2|}|\iint_{I\times J}\frac{|y_1 - z_1|^{\alpha}}{(t_1 + |x_1 - y_1|)^{n+\alpha}} \frac{|y_2 - z_2|^{\beta}}{(t_2 + |x_2 - y_2|)^{m+\beta}}|a(y_1,y_2)| dy_1dy_2|^2 \frac{dt_1}{t_1} \frac{dt_2}{t_2}\\
&\lesssim l_I^{2\alpha}l_J^{2\beta}\big(\int_{2l_I}^{\infty}\frac{dt_1}{t_1^{2n+2\alpha+1}}\big)
\int_{2|y_2-z_2|}^{|x_2-y_2|}|\iint_{I\times J}\frac{|x_2 - y_2|^{\varepsilon/2}}{t_2^{\varepsilon/2}|x_2 - y_2|^{m+\beta}}|a(y_1,y_2)| dy_1dy_2|^2\frac{dt_2}{t_2}\\
&\lesssim l_I^{-2n}l_J^{2\beta-\varepsilon}|x_2-z_2|^{-2m-2\beta+\varepsilon}.
\end{align*}
This leads to
\begin{align*}
B_5\lesssim l_I^{-n}l_J^{\beta-\varepsilon/2}\int_{2I}\int_{(\gamma J)^c}|x_2-z_2|^{-m-\beta+\varepsilon/2}dx_1dx_2\lesssim\gamma^{-\beta/2}.
\end{align*}
\noindent
\textbf {Estimate for $B_6$.} Note that $x\in E_2$ and $(t_1,t_2)\in G_6$, then for any $y_1\in I$, it holds that $t_1\geq 2|y_1-z_1|$ and $t_1>|x_1-y_1|$. By the support, vanishing and size condition of $H^1(\R^n\times\R^m)$ rectangle atom and H\"{o}lder estimate of $s_{t_1,t_2}$, we may get
\begin{align*}
&\int_0^{\infty}\int_0^{\infty}|\iint_{I\times J} \mathbf{1}_{G_6}(t_1,t_2)s_{t_1,t_2}(x,y))a(y_1,y_2) dy_1 \ dy_2|^2 \frac{dt_1}{t_1} \frac{dt_2}{t_2}\\
&=\int_0^{\infty}\int_0^{\infty}|\iint_{I\times J} \mathbf{1}_{F_{2,1}}(t_1,t_2)(s_{t_1,t_2}(x,y)
- s_{t_1,t_2}(x,(z_1,y_2)))a(y_1,y_2) dy_1 \ dy_2|^2  \frac{dt_1}{t_1} \frac{dt_2}{t_2}\\
&\lesssim \int_{2l_I}^{\infty}\int_{0}^{2|y_2-z_2|}|\iint_{I\times J}\frac{|y_1 - z_1|^{\alpha}}{(t_1 + |x_1 - y_1|)^{n+\alpha}} \frac{t_2^{\beta}}{(t_2 + |x_2 - y_2|)^{m+\beta}}|a(y_1,y_2)| dy_1dy_2|^2  \frac{dt_1}{t_1} \frac{dt_2}{t_2}\\
&\lesssim l_I^{2\alpha}|x_2-z_2|^{-2m-2\beta}\int_{2l_I}^{\infty}\int_{0}^{2|y_2-z_2|}\frac{t_2^{\beta}}{t_1^{2n+2\alpha+1}} dt_1dt_2
|\iint_{I\times J}|a(y_1,y_2)| dy_1dy_2|^2\\
&\lesssim l_I^{2n}l_J^{2\beta}|x_2-z_2|^{-2m-2\beta}.
\end{align*}
This yields that
\begin{align*}
B_6\lesssim l_I^{n}l_J^{\beta}\int_{2I}\int_{(\gamma J)^c}|x_2-z_2|^{-m-\beta}dx_1dx_2\lesssim\gamma^{-\beta}.
\end{align*}
Combining the estimate for $B_i$($i=1\cdot\cdot\cdot 6$), we get
\begin{align*}
&\iint_{E_{2,1}}|g(a)(x_1,x_2)|dx_1dx_2\lesssim \gamma^{-\beta/2}.
\end{align*}
Thus, we finish the proof of the inequality (\ref{R5}). Since we have reduced the proof of Theorem $\ref{thm1}$ to inequality (\ref{R5}), the proof of Theorem $\ref{thm1}$ is also completed.
\end{proof}
\section{Proof of Theorem 1.2}\label{Sec-300}
First, we list an elementary inequality which is useful in the proof of Theorem \ref{thm2}.
\begin{lemma}\label{lemma2}$($\cite{XY}$)$
Let $\theta_1, \theta_2>0$, $A\leq B$. Denote $\theta_0=\min\{\theta_1,\theta_2\}$, $\theta_3=\max\{\theta_1,\theta_2\}>n$, then
\begin{align*}
\int_{\R^n}\frac{dy}{(A+|y|)^{\theta_1}(B+|x-y|)^{\theta_2}}\lesssim A^{n-\theta_3}(|x|+B)^{-\theta_0}.
\end{align*}
\end{lemma}
\begin{proof}[Proof of Theorem~\ref{thm2}]
Similar to Theorem \ref{thm1}, in order to prove (\ref{Reg}), it suffices to verify that for any $H^1(\R^n\times\R^m)$-atom $\widetilde{a}$, there exists a constant $C>0$ satisfying
\begin{align}\label{Inq1}
\norm{g_{\vec{\lambda}}^*(\widetilde{a})}_{L^1(\R^n\times\R^m)}=\norm{T_2(\widetilde{a})}_{L^1_{\mathcal{H}_2}(\R^n\times\R^m)}\leq C.
 \end{align}
 In fact, let $f\in H^1(\R^n\times\R^m)$, by atomic decomposition, we have $f=\sum_j^{\infty}\lambda_j\widetilde{a}_j$, where $\widetilde{a}_j$ is $H^1(\R^n\times\R^m)$-atom. The size condition of $K_{t_1,t_2}$ and Fubini theorem lead to
 \begin{align}\label{in1}
 &\sum_{j=1}^{\infty}\lambda_j\iint_{\R^n\times\R^m}T_2(\widetilde{a}_j)(x_1, x_2)dx_1dx_2\\ \notag
 &\leq\sum_{j=1}^{\infty}\lambda_j\frac{1}{t_1^nt_2^m}\iint_{\R^{n}\times\R^m} \iint_{\R^{n}\times\R^m}
\Big(\frac{t_1}{t_1 + |x_1 - y_1|}\Big)^{\frac{n \lambda_1}{2}}\Big(\frac{t_2}{t_2 + |x_2 - y_2|}\Big)^{\frac{m \lambda_2}{2}} \\ \notag
&\quad\quad \times \iint_{\R^{n}\times\R^m}| K_{t_1,t_2}(y_1,y_2,z_1,z_2)| dy_1 dy_2 dx_1 dx_2\widetilde{a}(z_1,z_2) dz_1dz_2\\ \notag
 &\leq\sum_{j=1}^{\infty}\frac{\lambda_j}{t_1^nt_2^m}\iint_{\R^{n}\times\R^m} \iint_{\R^{n}\times\R^m}
\Big(\frac{t_1}{t_1 + |x_1 - y_1|}\Big)^{\frac{n \lambda_1}{2}}\Big(\frac{t_2}{t_2 + |x_2 - y_2|}\Big)^{\frac{m \lambda_2}{2}}\iint_{\R^{n}\times\R^m} \\ \notag
&\quad\quad \times \frac{t_1^{\alpha}}{(t_1 + |y_1 - z_1|)^{n+\alpha}} \frac{t_2^{\beta}}{(t_2 + |y_2 - z_2|)^{m+\beta}}dy_1 dy_2 dx_1 dx_2|\widetilde{a}(z_1,z_2)| dz_1dz_2.
 \end{align}
Since $0 < \alpha \leq n(\lambda_1 -2)/2$ and $0 < \beta \leq m(\lambda_2 -2)/2$, we may set $\varepsilon_1>0$, $\varepsilon_2>0$ with $2n+2\alpha\leq 2n+2\varepsilon_1\leq n\lambda_1$ and $2m+2\beta\leq 2m+2\varepsilon_2\leq m\lambda_2$, respectively. By Lemma \ref{lemma2}, we obtain
\begin{align}\label{in2}
&\int_{\R^n}\int_{\R^{n}}\Big(\frac{t_1}{t_1 + |x_1 - y_1|}\Big)^{\frac{n \lambda_1}{2}} \frac{t_1^{\alpha}}{(t_1 + |y_1 - z_1|)^{n+\alpha}}  dy_1 dx_1\\ \notag
&\lesssim \int_{\R^n}\int_{\R^{n}} \frac{t_1^{n+\varepsilon_1}}{\big(t_1 + |x_1 - y_1|\big)^{n+\varepsilon_1}} \frac{t_1^{\alpha}}{(t_1 + |y_1 - z_1|)^{n+\alpha}}  dy_1 dx_1\\ \notag
&\lesssim \int_{\R^n} \frac{t_1^{n+\alpha}}{(t_1 + |x_1 - z_1|)^{n+\alpha}}   dx_1\lesssim t_1^n.
\end{align}
Similarly, we have
\begin{align}\label{in3}
\int_{\R^m}\int_{\R^{m}}\Big(\frac{t_2}{t_2 + |x_2 - y_2|}\Big)^{\frac{m \lambda_2}{2}} \frac{t_2^{\beta}}{(t_2 + |y_2 - z_2|)^{m+\beta}}  dy_2 dx_2\lesssim t_2^m.
\end{align}
Putting (\ref{in2}) and (\ref{in3}) into (\ref{in1}), we have
 \begin{align}\label{in1}
 &\sum_{j=1}^{\infty}\lambda_j\iint_{\R^n\times\R^m}T_2(\widetilde{a}_j)(x_1, x_2)dx_1dx_2\leq C.
 \end{align}
 Hence, (\ref{in1}) and Lebesgue dominated convergence theorem imply that
 \begin{equation*}
 T_2(f)=\sum_{j=1}^{\infty}\lambda_jT_2(\widetilde{a}_j),\ \ a. e.
 \end{equation*}
 By (\ref{Inq1}) and repeating the same steps as in (\ref{iin9}), we have
 \begin{align*}
 \norm{T_2(f)}_{L^1_{\mathcal{H}_2}(\R^n\times\R^m)}\lesssim\norm{f}_{H^1(\R^n\times\R^m)}.
 \end{align*}
 Thus, to prove Theorem \ref{thm2}, we only need to prove inequality (\ref{Inq1}).\par
From now on, we are devoted to prove (\ref{Inq1}).\par
Let $a$ be any $H^1(\R^n\times\R^m)$ rectangle atom, with support on a rectangle $R=I\times J$. As remarked before, $g_{\vec{\lambda}}^*(f)$ is bounded from $L^2(\R^n\times\R^m)$ to $L^2_{\mathcal{H}_2}(\R^n\times\R^m)$. Thus, by Lemma \ref{lemma}, in order to prove $(\ref{Reg})$, it suffices to verify that 
\begin{align}\label{Re1g}
\iint_{\R^n\times\R^m\setminus\widetilde{R}_{\gamma}}|g_{\vec{\lambda}}^*(a)(x_1,x_2)|dx_1dx_2\lesssim\gamma^{-\delta}, \text {\ \ \quad for all $\gamma\geq 2$}.
\end{align}
Recall the definition of $E_i$ in the proof of Theorem \ref{thm1} and by symmetry, to prove (\ref{Re1g}), we only need to show that there exist a $\delta>0$ such that
\begin{align}\label{Re2g}
\iint_{E_1}|g_{\vec{\lambda}}^*(a)(x_1,x_2)|dx_1dx_2\lesssim\gamma^{-\delta} \text {\ \ for all $\gamma\geq 2$}
\end{align}
and
\begin{align}\label{Re3g}
\iint_{E_2}|g_{\vec{\lambda}}^*(a)(x_1,x_2)|dx_1dx_2\lesssim\gamma^{-\delta} \text {\ \ for all $\gamma\geq 2$}.
\end{align}

Now, let us to begin with the proof of (\ref{Re2g}).\\
\noindent
 \bull
\textbf {Proof of (\ref{Re2g}).}
By the definition of $g_{\vec{\lambda}}^*$ and the support condition of $H^1(\R^n\times\R^m)$ rectangle atom, we get
\begin{align*}
|g_{\vec{\lambda}}^*(a)(x)|^2
&= \iint_{\R^{m+1}_{+}} \Big(\frac{t_2}{t_2 + |x_2 - y_2|}\Big)^{m \lambda_2}
\iint_{\R^{n+1}_{+}} \Big(\frac{t_1}{t_1 + |x_1 - y_1|}\Big)^{n \lambda_1} \\
&\quad\quad \times |\iint_{\R^{n}\times\R^m} K_{t_1,t_2}(y_1,y_2,z_1,z_2)a(z_1,z_2) dz_1 \ dz_2|^2 \frac{dy_1 dt_1}{t_1^{n+1}} \frac{dy_2 dt_2}{t_2^{m+1}}.
\end{align*}With abuse of notation, we split the domain of variable $t_1$ and $t_2$ as follows,
\begin{align*}
(t_1,t_2)\in \R^+\times\R^+\subset&\big([2l_I,\infty)\times[2l_J,\infty)\big)\cup\big([2l_I,\infty)\times(0,2l_J)\big)\\
&\cup\big((0,2l_I)\times[2l_J,\infty)\big)\cup\big((0,2l_I)\times(0,2l_J)\big)\\
&=:F_1\cup F_2\cup F_3\cup F_4.
\end{align*}
Hence, for any $x\in E_1$, we have
\begin{align*}
|g_{\vec{\lambda}}^*(a)(x)|^2&\leq \sum_{i=1}^{4}
\iint_{\R^{m+1}_{+}} \Big(\frac{t_2}{t_2 + |x_2 - y_2|}\Big)^{m \lambda_2}
\iint_{\R^{n+1}_{+}} \Big(\frac{t_1}{t_1 + |x_1 - y_1|}\Big)^{n \lambda_1} \mathbf{1}_{F_i}(t_1,t_2)\\
&\quad\quad  |\iint_{\R^{n}\times\R^m} K_{t_1,t_2}(y_1,y_2,z_1,z_2)a(z_1,z_2) dz_1 \ dz_2|^2 \frac{dy_1 dt_1}{t_1^{n+1}} \frac{dy_2 dt_2}{t_2^{m+1}}\\
&=:\sum_{i=1}^4A_i.
\end{align*}
\quad Denote $z_1'$ as the centre of $I$ and $z_2'$ as the centre of $J$. Now, we will estimate each $A_i$.\par
\noindent
\textbf {Estimate for $A_1$.} Since $x\in E_1$ and $(t_1,t_2)\in F_1$, then for any $y_1\in I$, $y_2\in J$, we have $t_1\geq 2l_I\geq2|z_1-z_1'|$ and $t_2\geq 2l_J\geq2|z_2-z_2'|$. By the vanishing, support and size condition of $H^1(\R^n\times\R^m)$ rectangle atom, H\"{o}lder estimate of $K_{t_1,t_2}$ and H\"{o}lder inequality, we have
\begin{align}\label{A1}
A_1&=\iint_{\R^{m+1}_{+}} \Big(\frac{t_2}{t_2 + |x_2 - y_2|}\Big)^{m \lambda_2}
\iint_{\R^{n+1}_{+}} \Big(\frac{t_1}{t_1 + |x_1 - y_1|}\Big)^{n \lambda_1}\mathbf{1}_{F_1}(t_1,t_2) \\ \notag
&\quad\quad  \times|\iint_{I\times J} (K_{t_1,t_2}(y,z) - K_{t_1,t_2}(y,(z_1,z_2'))\\ \notag
 &\quad- K_{t_1,t_2}(y,(z_1',z_2)) + K_{t_1,t_2}(y,(z_1',z_2')))a(z_1,z_2) dz_1 dz_2|^2 \frac{dy_1 dt_1}{t_1^{n+1}} \frac{dy_2 dt_2}{t_2^{m+1}}\\ \notag
&\lesssim \int_{2l_J}^{\infty}\int_{\R^{m}} \Big(\frac{t_2}{t_2 + |x_2 - y_2|}\Big)^{m \lambda_2}
\int_{2l_I}^{\infty}\int_{\R^{n}} \Big(\frac{t_1}{t_1 + |x_1 - y_1|}\Big)^{n \lambda_1} |\iint_{I\times J}\\ \notag
&\quad\quad \frac{|z_1 - z_1'|^{\alpha}}{(t_1 + |y_1 - z_1|)^{n+\alpha}} \frac{|z_2 - z_2'|^{\beta}}{(t_2 + |y_2 - z_2|)^{m+\beta}} a(z_1,z_2) dz_1 \ dz_2|^2 \frac{dy_1 dt_1}{t_1^{n+1}} \frac{dy_2 dt_2}{t_2^{m+1}}\\ \notag
&\lesssim
\norm{a}_2^2\int_{2l_I}^{\infty}\int_{\R^{n}}\int_{I} \Big(\frac{t_1}{t_1 + |x_1 - y_1|}\Big)^{n \lambda_1} \frac{|z_1 - z_1'|^{2\alpha}}{(t_1 + |y_1 - z_1|)^{2n+2\alpha}}  dz_1  \frac{dy_1 dt_1}{t_1^{n+1}} \\ \notag
&\quad\times \int_{2l_J}^{\infty}\int_{\R^{m}}\int_{J} \Big(\frac{t_2}{t_2 + |x_2 - y_2|}\Big)^{m \lambda_2} \frac{|z_2 - z_2'|^{2\beta}}{(t_2 + |y_2 - z_2|)^{2m+2\beta}}  dz_2  \frac{dy_2 dt_2}{t_2^{m+1}}.
\end{align}
By Lemma \ref{lemma2}, it holds that
\begin{align}\label{Ak}
&\int_{2l_I}^{\infty}\int_{\R^{n}}\int_{I} \Big(\frac{t_1}{t_1 + |x_1 - y_1|}\Big)^{n \lambda_1} \frac{|z_1 - z_1'|^{2\alpha}}{(t_1 + |y_1 - z_1|)^{2n+2\alpha}}  dz_1  \frac{dy_1 dt_1}{t_1^{n+1}}\\ \notag
&\lesssim l_I^{2\alpha}\int_{2l_I}^{\infty}\int_{I}\int_{\R^{n}} \frac{t_1^{2n+2\varepsilon_1}}{\big(t_1 + |x_1 - y_1|\big)^{2n+2\varepsilon_1}} \frac{1}{(t_1 + |y_1 - z_1|)^{2n+2\alpha}}  dy_1  \frac{dz_1 dt_1}{t_1^{n+1}}\\ \notag
&\lesssim l_I^{2\alpha}\int_{2l_I}^{\infty}\int_{I}\frac{t_1^n}{\big(t_1 + |x_1 - z_1|\big)^{2n+2\alpha}} \frac{dz_1 dt_1}{t_1^{n+1}}
\lesssim l_I^{n+2\alpha}|x_1-z_1|^{-2n-2\alpha}.
\end{align}
Similarly, we may get
\begin{align}\label{Akk}
\int_{2l_J}^{\infty}\int_{\R^{m}}\int_{J} \Big(\frac{t_2}{t_2 + |x_2 - y_2|}\Big)^{m \lambda_2} &\frac{|z_2 - z_2'|^{2\beta}}{(t_2 + |y_2 - z_2|)^{2n+2\beta}}  dz_2  \frac{dy_2 dt_2}{t_2^{m+1}}\\ \notag
&\lesssim l_J^{m+2\beta}|x_2-z_2|^{-2m-2\beta}.
\end{align}
Putting the above estimate into (\ref{A1}) and recall that $\norm{a}_2^2\lesssim |R|^{-1}=l_I^{-n}l_J^{-m}$, we have
\begin{align*}
A_1\lesssim l_I^{2\alpha}l_J^{2\beta}|x_1-z_1|^{-2n-2\alpha}|x_2-z_2|^{-2m-2\beta}.
\end{align*}
\noindent
\textbf {Estimate for $A_2$.} Since $x\in E_1$ and $(t_1,t_2)\in F_2$, then for any $y_1\in I$, we have $t_1\geq 2l_I\geq2|z_1-z_1'|$. By the vanishing, support and size condition of $H^1(\R^n\times\R^m)$ rectangle atom, mixed H\"{o}lder and size condition of $K_{t_1,t_2}$ and H\"{o}lder inequality, we have
\begin{align}\label{A2}
A_2&=\iint_{\R^{m+1}_{+}} \Big(\frac{t_2}{t_2 + |x_2 - y_2|}\Big)^{m \lambda_2}
\iint_{\R^{n+1}_{+}} \Big(\frac{t_1}{t_1 + |x_1 - y_1|}\Big)^{n \lambda_1} \mathbf{1}_{F_2}(t_1,t_2)\\ \notag
&\quad\quad  |\iint_{I\times J} (K_{t_1,t_2}(y,z) - K_{t_1,t_2}(y,(z_1',z_2)))a(z_1,z_2) dz_1 \ dz_2|^2 \frac{dy_1 dt_1}{t_1^{n+1}} \frac{dy_2 dt_2}{t_2^{m+1}}\\ \notag
&\lesssim \int_{0}^{2l_J}\int_{\R^{m}} \Big(\frac{t_2}{t_2 + |x_2 - y_2|}\Big)^{m \lambda_2}
\int_{2l_I}^{\infty}\int_{\R^{n}} \Big(\frac{t_1}{t_1 + |x_1 - y_1|}\Big)^{n \lambda_1}|\iint_{I\times J}\\ \notag
&\quad\quad  \frac{|z_1 - z_1'|^{\alpha}}{(t_1 + |y_1 - z_1|)^{n+\alpha}} \frac{t_2^{\beta}}{(t_2 + |y_2 - z_2|)^{m+\beta}} a(z_1,z_2) dz_1 \ dz_2|^2 \frac{dy_1 dt_1}{t_1^{n+1}} \frac{dy_2 dt_2}{t_2^{m+1}}\\ \notag
&\lesssim
\norm{a}_2^2\int_{2l_I}^{\infty}\int_{\R^{n}}\int_{I} \Big(\frac{t_1}{t_1 + |x_1 - y_1|}\Big)^{n \lambda_1} \frac{|z_1 - z_1'|^{2\alpha}}{(t_1 + |y_1 - z_1|)^{2n+2\alpha}}  dz_1  \frac{dy_1 dt_1}{t_1^{n+1}} \\ \notag
&\quad\times \int_{0}^{2l_J}\int_{\R^{m}}\int_{J} \Big(\frac{t_2}{t_2 + |x_2 - y_2|}\Big)^{m \lambda_2} \frac{t_2^{2\beta}}{(t_2 + |y_2 - z_2|)^{2m+2\beta}}  dz_2  \frac{dy_2 dt_2}{t_2^{m+1}}.
\end{align}
Note that $2m+2\beta\leq 2m+2\varepsilon_2\leq m\lambda_2$. By Lemma \ref{lemma2}, we obtain
\begin{align}\label{Ak11}
\int_{0}^{2l_J}&\int_{\R^{m}}\int_{J} \Big(\frac{t_2}{t_2 + |x_2 - y_2|}\Big)^{m \lambda_2} \frac{t_2^{2\beta}}{(t_2 + |y_2 - z_2|)^{2m+2\beta}}  dz_2  \frac{dy_2 dt_2}{t_2^{m+1}}\\ \notag
&\lesssim \int_{0}^{2l_J}\int_{I}\int_{\R^{m}} \frac{t_2^{2m+2\varepsilon_2+2\beta}}{\big(t_2 + |x_2 - y_2|\big)^{2m+2\varepsilon_2}} \frac{1}{(t_2 + |y_2 - z_2|)^{2m+2\beta}}  dy_2  \frac{dz_2 dt_2}{t_2^{m+1}}\\ \notag
&\lesssim \int_{0}^{2l_J}\int_{J}\frac{t_2^{2\beta}}{\big(t_2 + |x_2 - z_2|\big)^{2m+2\beta}} \frac{dz_2 dt_2}{t_2^{m+1}}\lesssim l_J^{m+2\beta}|x_2-z_2|^{-2m-2\beta}.
\end{align}
Putting inequality (\ref{Ak}) and (\ref{Ak11}) into inequality (\ref{A2}), we have
\begin{align*}
A_2\lesssim l_I^{2\alpha}l_J^{2\beta}|x_1-z_1|^{-2n-2\alpha}|x_2-z_2|^{-2m-2\beta}.
\end{align*}
\textbf {Estimate for $A_3$.} $A_3$ is symmetric with $A_2$, thus,
\begin{align*}
A_3\lesssim l_I^{2\alpha}l_J^{2\beta}|x_1-z_1|^{-2n-2\alpha}|x_2-z_2|^{-2m-2\beta}.
\end{align*}
\textbf {Estimate for $A_4$.} Similar to $A_1$ and $A_2$, the vanishing, support and size condition of rectangle atom, size condition of $K_{t_1,t_2}$, H\"{o}lder inequality and inequality (\ref{Ak11}) lead to
\begin{align*}
A_4\lesssim l_I^{2\alpha}l_J^{2\beta}|x_1-z_1|^{-2n-2\alpha}|x_2-z_2|^{-2m-2\beta}.
\end{align*}
Hence, whenever $x\in E_1$, we get
\begin{align*}
|g_{\vec{\lambda}}^*(a)(x)|^2&\leq\sum_{i=1}^4A_i\lesssim l_I^{2\alpha}l_J^{2\beta}|x_1-z_1|^{-2n-2\alpha}|x_2-z_2|^{-2m-2\beta}.
\end{align*}
This yields that
\begin{align*}
\iint_{E_1}|g_{\vec{\lambda}}^*(a)(x_1,x_2)|dx_1dx_2\lesssim\gamma^{-\alpha-\beta} .
\end{align*}
Thus, we have proved inequality (\ref{Re2}).

Now, we are in the position to prove inequality (\ref{Re3}).\\
\noindent
 \bull
\textbf {Proof of (\ref{Re3}).}
Recall that $E_2=\{(x_1,x_2)\in\R^n\times\R^m: x_1\in\gamma_1 I,x_2\notin\gamma J\}$, we have
\begin{align*}
E_2&\subset\{(x_1,x_2)\in\R^n\times\R^m: x_1\in 2I,x_2\notin\gamma J\}\\
&\quad\cup\{(x_1,x_2)\in\R^n\times\R^m: x_1\in\gamma_1 I\backslash 2I,x_2\notin\gamma J\}=:E_{2,1}\cup E_{2,2}.
\end{align*}
Therefore,
\begin{align*}
\iint_{E_2}|g_{\vec{\lambda}}^*(a)(x_1,x_2)|dx_1dx_2&\lesssim\iint_{E_{2,1}}|g_{\vec{\lambda}}^*(a)(x_1,x_2)|dx_1dx_2\\&\quad\quad+
\iint_{E_{2,2}}|g_{\vec{\lambda}}^*(a)(x_1,x_2)|dx_1dx_2 .
\end{align*}
Since $E_{2,2}\subset\{(x_1,x_2)\in\R^n\times\R^m: x_1\notin2I,x_2\notin\gamma J\}$, taking $\gamma_1=2$ and repeating the proof of (\ref{Re2}), we may get
\begin{align*}
\iint_{E_{2,2}}|g_{\vec{\lambda}}^*(a)(x_1,x_2)|dx_1dx_2\lesssim\gamma^{-\alpha-\beta} .
\end{align*}
So, to prove Theorem \ref{Re}, we only need to show that there exists a $\delta>0$ such that
\begin{align}\label{R52}
\iint_{E_{2,1}}|g_{\vec{\lambda}}^*(a)(x_1,x_2)|dx_1dx_2\lesssim\gamma^{-\delta} .
\end{align}
We also need to split the domain of variable $t_1$ and $t_2$ as follows,
\begin{align*}
(t_1,t_2)\in \R^+\times\R^+\subset&\big([2l_I,\infty)\times[2l_J,\infty)\big)\cup\big([2l_I,\infty)\times(0,2l_J)\big)\\
&\cup\big((0,2l_I)\times[2l_J,\infty)\big)\cup\big((0,2l_I)\times(0,2l_J)\big)\\
&=F_1\cup F_2\cup F_3\cup F_4.
\end{align*}
Hence, we may write
\begin{align*}
&\iint_{E_{2,1}}|g_{\vec{\lambda}}^*(a)(x_1,x_2)|dx_1dx_2\\
&\lesssim\sum_{i=1}^4\iint_{E_{2,1}}\bigg(\iint_{\R^{m+1}_{+}} \Big(\frac{t_2}{t_2 + |x_2 - y_2|}\Big)^{m \lambda_2}
\iint_{\R^{n+1}_{+}} \Big(\frac{t_1}{t_1 + |x_1 - y_1|}\Big)^{n \lambda_1}\mathbf{1}_{F_i}(t_1,t_2)  \\
&\quad\quad \times |\iint_{\R^{n}\times\R^m} K_{t_1,t_2}(y_1,y_2,z_1,z_2)a(z_1,z_2) dz_1 \ dz_2|^2 \frac{dy_1 dt_1}{t_1^{n+1}} \frac{dy_2 dt_2}{t_2^{m+1}}\bigg)^{1/2}dx_1dx_2\\
&=:\sum_{i=1}^4B_i.
\end{align*}
Now, we need to give the estimate for each $B_i$.\\
\noindent
\textbf {Estimate for $B_1$.} Since $x\in E_{2,1}$ and $(t_1,t_2)\in F_1$, then for any $z_1\in I$, $z_2\in J$, we have $t_1\geq 2l_I\geq2|z_1-z_1'|$ and $t_2\geq 2l_J\geq2|z_2-z_2'|$. Similar to inequality (\ref{A1}), by the vanishing, support and size condition of $H^1(\R^n\times\R^m)$ rectangle atom, H\"{o}lder estimate of $K_{t_1,t_2}$ and H\"{o}lder inequality, we have
\begin{align}\label{B1}
B_1
&\lesssim
\norm{a}_2^2\iint_{E_{2,1}}\bigg(\int_{2l_I}^{\infty}\int_{\R^{n}}\int_{I} \Big(\frac{t_1}{t_1 + |x_1 - y_1|}\Big)^{n \lambda_1} \frac{|z_1 - z_1'|^{2\alpha}}{(t_1 + |y_1 - z_1|)^{2n+2\alpha}}  dz_1  \frac{dy_1 dt_1}{t_1^{n+1}} \\ \notag
&\quad\times \int_{2l_J}^{\infty}\int_{\R^{m}}\int_{J} \Big(\frac{t_2}{t_2 + |x_2 - y_2|}\Big)^{m \lambda_2} \frac{|z_2 - z_2'|^{2\beta}}{(t_2 + |y_2 - z_2|)^{2m+2\beta}}  dz_2  \frac{dy_2 dt_2}{t_2^{m+1}}\bigg)^{1/2}dx_1dx_2.
\end{align}
Since $x\in E_{2,1}$, then we have $t_1\geq 2|x_1-z_1|$. By Lemma \ref{lemma2}, we obtain
\begin{align}\label{Ak21}
&\int_{2l_I}^{\infty}\int_{\R^{n}}\int_{I} \Big(\frac{t_1}{t_1 + |x_1 - y_1|}\Big)^{n \lambda_1} \frac{|z_1 - z_1'|^{2\alpha}}{(t_1 + |y_1 - z_1|)^{2n+2\alpha}}  dz_1  \frac{dy_1 dt_1}{t_1^{n+1}}\\ \notag
&\lesssim l_I^{2\alpha}\int_{2l_I}^{\infty}\int_{I}\int_{\R^{n}} \frac{t_1^{2n+2\varepsilon_1}}{\big(t_1 + |x_1 - y_1|\big)^{2n+2\varepsilon_1}} \frac{1}{(t_1 + |y_1 - z_1|)^{2n+2\alpha}}  dy_1  \frac{dz_1 dt_1}{t_1^{n+1}}\\ \notag
&\lesssim l_I^{2\alpha}\int_{2l_I}^{\infty}\int_{I}\frac{t_1^n}{\big(t_1 + |x_1 - z_1|\big)^{2n+2\alpha}} \frac{dz_1 dt_1}{t_1^{n+1}}\lesssim l_I^{n}.
\end{align}
Similarly, we may get
\begin{align}\label{Akk}
\int_{2l_J}^{\infty}\int_{\R^{m}}\int_{J} \Big(\frac{t_2}{t_2 + |x_2 - y_2|}\Big)^{m \lambda_2} &\frac{|z_2 - z_2'|^{2\beta}}{(t_2 + |y_2 - z_2|)^{2n+2\beta}}  dz_2  \frac{dy_2 dt_2}{t_2^{m+1}}\\ \notag
&\lesssim l_J^{m+2\beta}|x_2-z_2|^{-2m-2\beta}.
\end{align}
Putting the inequality (\ref{Ak21}), (\ref{Akk}) into inequality (\ref{B1}), recall that $\norm{a}_2^2\lesssim |R|^{-1}=l_I^{-n}l_J^{-m}$, we may arrive at
$
B_1\lesssim \gamma^{-\beta}.
$

\noindent
\textbf {Estimate for $B_2$.}
Since $x\in E_{2,1}$ and $(t_1,t_2)\in F_2$, for any $z_1\in I$, we have $t_1\geq 2l_I\geq2|z_1-z_1'|$. By the vanishing, support and size condition of rectangle atom, mixed H\"{o}lder and size condition of $K_{t_1,t_2}$ and H\"{o}lder inequality, similar as (\ref{A2}), we have
\begin{align}\label{Ak23}
B_2
&\lesssim
\norm{a}_2^2\iint_{E_{2,1}}\bigg(\int_{2l_I}^{\infty}\int_{\R^{n}}\int_{I} \Big(\frac{t_1}{t_1 + |x_1 - y_1|}\Big)^{n \lambda_1} \frac{|z_1 - z_1'|^{2\alpha}}{(t_1 + |y_1 - z_1|)^{2n+2\alpha}}  dz_1  \frac{dy_1 dt_1}{t_1^{n+1}} \\ \notag
&\quad\times \int_{0}^{2l_J}\int_{\R^{m}}\int_{J} \Big(\frac{t_2}{t_2 + |x_2 - y_2|}\Big)^{m \lambda_2} \frac{t_2^{2\beta}}{(t_2 + |y_2 - z_2|)^{2m+2\beta}}  dz_2  \frac{dy_2 dt_2}{t_2^{m+1}}\bigg)^{1/2}dx_1dx_2.
\end{align}
Putting inequalities (\ref{Ak21}) and (\ref{Ak11}) into (\ref{Ak23}), we may get
$
B_2\lesssim \gamma^{-\beta}.
$\\
\noindent
\textbf {Estimate for $B_3$.} If $x\in E_{2,1}$, $t\in F_3$, then for any $z_2\in J$, we have $t_2\geq 2l_J\geq2|z_2-z_2'|$. The vanishing, support and size condition of $H^1(\R^n\times\R^m)$ rectangle atom, combination of Carleson and H\"{o}lder conditions of $K_{t_1,t_2}$, H\"{o}lder inequality and (\ref{Ak11}) imply that
\begin{align*}
B_3&=\iint_{E_{2,1}}\bigg(\iint_{\R^{m+1}_{+}} \Big(\frac{t_2}{t_2 + |x_2 - y_2|}\Big)^{m \lambda_2}
\iint_{\R^{n+1}_{+}} \Big(\frac{t_1}{t_1 + |x_1 - y_1|}\Big)^{n \lambda_1}\mathbf{1}_{F_3}(t_1,t_2) \\
&\quad\times |\iint_{\R^{n}\times\R^m} (K_{t_1,t_2}(y, z)-K_{t_1,t_2}(y, (z_1,z_2')))a(z_1,z_2) dz_1dz_2|^2 \frac{dy_1 dt_1}{t_1^{n+1}} \frac{dy_2 dt_2}{t_2^{m+1}}\bigg)^{\frac12}dx_1dx_2\\
&\lesssim\norm{a}_{L^2}\iint_{2I\times(\gamma J)^c}\bigg(\iint_{\R^{m+1}_{+}} \Big(\frac{t_2}{t_2 + |x_2 - y_2|}\Big)^{m \lambda_2}
\iint_{\R^{n+1}_{+}} \Big(\frac{t_1}{t_1 + |x_1 - y_1|}\Big)^{n \lambda_1}\mathbf{1}_{F_2}(t_1,t_2) \\
&\quad\quad \iint_{\R^{n}\times\R^m} |K_{t_1,t_2}(y, z)-K_{t_1,t_2}(y, (z_1,z_2'))|^2 dz_1 \ dz_2\frac{dy_1 dt_1}{t_1^{n+1}} \frac{dy_2 dt_2}{t_2^{m+1}}\bigg)^{1/2}dx_1dx_2\\
&\lesssim\norm{a}_{L^2}|2I|^{1/2}\int_{(\gamma J)^c}\bigg(\int_{2l_j}^{\infty}\int_{\R^{m}}\int_J\Big(\frac{t_2}{t_2 + |x_2 - y_2|}\Big)^{m \lambda_2}
\big(\iint_{\widehat{I}}\int_{\R^n} \Big(\frac{t_1}{t_1 + |x_1 - y_1|}\Big)^{n \lambda_1} \\
&\quad\quad  |K_{t_1,t_2}(y, z)-K_{t_1,t_2}(y, (z_1,z_2'))|^2 dz_1 \frac{dy_1 dx_1dt_1}{t_1^{n+1}}\big) \frac{dz_2dy_2 dt_2}{t_2^{m+1}}\bigg)^{1/2}dx_2\\
&\lesssim\norm{a}_{L^2}|I|^{\frac12}\int_{(\gamma J)^c}\bigg(\int_{2l_j}^{\infty}\int_{\R^{m}}\int_J\Big(\frac{t_2}{t_2 + |x_2 - y_2|}\Big)^{m \lambda_2}
\frac{|z_2 - z_2'|^{\beta}}{(t_2 + | y_2 - z_2|)^{m + \beta}} \frac{dz_2dy_2 dt_2}{t_2^{m+1}}\bigg)^{\frac12}dx_2\\
&\lesssim\norm{a}_{L^2}|I|^{\frac12}l_J^{m/2+\beta}\int_{(\gamma J)^c}|x_2-z_2|^{-m-\beta}dx_2\lesssim \gamma^{-\beta}.
\end{align*}
\textbf {Estimate for $B_4$.} Similarly, we have
\begin{align*}
&B_4=\iint_{E_{2,1}}\bigg(\iint_{\R^{m+1}_{+}} \Big(\frac{t_2}{t_2 + |x_2 - y_2|}\Big)^{m \lambda_2}
\iint_{\R^{n+1}_{+}} \Big(\frac{t_1}{t_1 + |x_1 - y_1|}\Big)^{n \lambda_1}\mathbf{1}_{F_4}(t_1,t_2) \\
&\quad\quad |\iint_{\R^{n}\times\R^m} K_{t_1,t_2}(y, z)a(z_1,z_2) dz_1 \ dz_2|^2 \frac{dy_1 dt_1}{t_1^{n+1}} \frac{dy_2 dt_2}{t_2^{m+1}}\bigg)^{1/2}dx_1dx_2\\
&\lesssim\norm{a}_{L^2}\iint_{2I\times(\gamma J)^c}\bigg(\iint_{\R^{m+1}_{+}} \Big(\frac{t_2}{t_2 + |x_2 - y_2|}\Big)^{m \lambda_2}
\iint_{\R^{n+1}_{+}} \Big(\frac{t_1}{t_1 + |x_1 - y_1|}\Big)^{n \lambda_1}\mathbf{1}_{F_4}(t_1,t_2) \\
&\quad\quad \iint_{\R^{n}\times\R^m} |K_{t_1,t_2}(y, z)|^2 dz_1 \ dz_2\frac{dy_1 dt_1}{t_1^{n+1}} \frac{dy_2 dt_2}{t_2^{m+1}}\bigg)^{1/2}dx_1dx_2 \end{align*}\begin{align*}{}&\lesssim\norm{a}_{L^2}(2I)^{1/2}\int_{(\gamma J)^c}\bigg(\int_{0}^{2l_j}\int_{\R^{m}}\int_J\Big(\frac{t_2}{t_2 + |x_2 - y_2|}\Big)^{m \lambda_2}
\big(\iint_{\widehat{I}}\int_{\R^n} \Big(\frac{t_1}{t_1 + |x_1 - y_1|}\Big)^{n \lambda_1}\\
&\quad\quad  |K_{t_1,t_2}(y, z)|^2 dz_1 \frac{dy_1 dx_1dt_1}{t_1^{n+1}}\big) \frac{dz_2dy_2 dt_2}{t_2^{m+1}}\bigg)^{1/2}dx_2\\
&\lesssim\norm{a}_{L^2}|I|^{\frac12}\int_{(\gamma J)^c}\bigg(\int_{2l_j}^{\infty}\int_{\R^{m}}\int_J\Big(\frac{t_2}{t_2 + |x_2 - y_2|}\Big)^{m \lambda_2}
\frac{t_2^{\beta}}{(t_2 + | y_2 - z_2|)^{m + \beta}} \frac{dz_2dy_2 dt_2}{t_2^{m+1}}\bigg)^{\frac12}dx_2\\
&\lesssim\norm{a}_{L^2}(2I)^{1/2}l_J^{m/2+\beta}\int_{(\gamma J)^c}|x_2-z_2|^{-m-\beta}dx_2\lesssim \gamma^{-\beta}.
\end{align*}
Combining the estimate for $B_i$($i=1\cdot\cdot\cdot 4$), we get
\begin{align*}
&\iint_{E_{2,1}}|g_{\vec{\lambda}}^*(a)(x_1,x_2)|dx_1dx_2\lesssim \gamma^{-\beta}.
\end{align*}
This completes the proof of (\ref{R52}). Since we have reduced the proof of Theorem $\ref{thm2}$ to inequality (\ref{R52}), the proof of Theorem $\ref{thm2}$ is finished.
\end{proof}
\section{Proof of the corollaries}\label{Sec-3000}
We need the following lemma:
\begin{thmC}$($Calder\'{o}n-Zygmund Lemma \cite{CF2}$)$\label{lemma3}
Let $\alpha>0$ be given and $f\in L^p(\R^n\times\R^m)$, $1 < p< 2$. Then we may write $f = f_1 + f_2$ where $f_1\in L^2(\R^n\times\R^m)$ and $f_2\in H^1(\R^n\times\R^m)$ with $\norm{f_1}_2^2\leq\alpha^{2-p}\norm{f}_{L_p^p}$; and $\norm{f_2}_{H^1}\leq c\alpha^{1-p}\norm{f}_{L^p_p}$, where $c$ is a universal constant.
\end{thmC}
A trivial corollary of Theorem C will lead to the following result.
\begin{lemma}\label{lemma4}
Let $T$ be a sublinear operator which is bounded from $H^1(\R^n\times\R^m)$ to $L^1(\R^n\times\R^m)$ and bounded on $L^2(\R^n\times\R^m)$. Then $T$ is bounded on $L^p(\R^n\times\R^m)$ for all $1<p<2$.
\end{lemma}

Thus, the $L^p$ boundedness in Corollary 1 and Corollary 2 follows from the interpolation between $(H^1, L^1)$ and $(L^2, L^2)$.

\end{document}